\documentclass[twoside, a4paper,10pt]{amsart}

\usepackage[pdftex,pdftitle={DG-algebras and derived A-infinity
  algebras},pdfauthor={Steffen Sagave},pdfstartview={FitH}]{hyperref}
\usepackage{amsmath}  
\usepackage{amsfonts}
\usepackage{amsthm}  
\usepackage{amssymb}
\usepackage{enumerate}
\usepackage[all]{xy}
\SelectTips{cm}{}
\SilentMatrices

\numberwithin{equation}{section}
\theoremstyle{plain}
\newtheorem{lemma}[equation]{Lemma}
\newtheorem{theorem}[equation]{Theorem}

\newtheorem{corollary}[equation]{Corollary}
\newtheorem{proposition}[equation]{Proposition}
\theoremstyle{definition}

\newtheorem{definition}[equation]{Definition}
\newtheorem{example}[equation]{Example}
\newtheorem{remark}[equation]{Remark}

\newcommand{\mF}{{\mathbb F}}

\newcommand{\mN}{{\mathbb N}}
\newcommand{\mZ}{{\mathbb Z}}

\newcommand{\cC}{{\mathcal C}}
\newcommand{\cD}{{\mathcal D}}

\newcommand{\cP}{{\mathcal P}}
\newcommand{\cS}{{\mathcal S}}

\DeclareMathOperator{\id}{id}
\DeclareMathOperator{\Hom}{Hom}
\DeclareMathOperator{\Mod}{Mod}
\DeclareMathOperator{\Ho}{Ho}

\DeclareMathOperator{\kernel}{ker}
\DeclareMathOperator{\image}{im}
\DeclareMathOperator{\Ext}{Ext}
\DeclareMathOperator{\Mat}{Mat}

\DeclareMathOperator{\enrhom}{\underline{Hom}}
\DeclareMathOperator{\homsp}{hom}
\DeclareMathOperator{\HH}{HH}
\DeclareMathOperator{\dHH}{dHH}
\DeclareMathOperator{\Chxes}{Ch}
\DeclareMathOperator{\biChxes}{biCh}
\DeclareMathOperator{\tChxes}{tCh}
\DeclareMathOperator{\dga}{dga}
\DeclareMathOperator{\bidga}{bidga}
\DeclareMathOperator{\tdga}{tdga}
\DeclareMathOperator{\sChxes}{sCh}
\DeclareMathOperator{\sdga}{sdga}

\DeclareMathOperator{\Tot}{Tot}
\DeclareMathOperator{\res}{res}

\newcommand{\eins}{\mathbf{1}}
\newcommand{\tensor}{\otimes}
\newcommand{\ovl}{\overline}

\newcommand{\ot}{\leftarrow}
\newcommand{\iso}{\cong}

\newcommand{\op}{{\mathrm{op}}}
\newcommand{\Tbar}{\ovl{T}}
\newcommand{\sclprod}[2]{\langle #1 , #2 \rangle }
\newcommand{\bimodu}[1]{\textrm{biMod-} #1}

\newcommand{\dtil}{\widetilde{d}}
\newcommand{\ftil}{\widetilde{f}}
\newcommand{\gtil}{\widetilde{g}}
\newcommand{\htil}{\widetilde{h}}
\newcommand{\mtil}{\widetilde{m}}
\newcommand{\ytil}{\widetilde{y}}
\newcommand{\ztil}{\widetilde{z}}
\newcommand{\alphatil}{\widetilde{\alpha}}
\newcommand{\betatil}{\widetilde{\beta}}
\newcommand{\mutil}{\widetilde{\mu}}
\newcommand{\psitil}{\widetilde{\psi}}
\newcommand{\rhotil}{\widetilde{\rho}}
\newcommand{\fgtil}{\widetilde{fg}}

\newcommand{\gbetatil}{(\widetilde{g \beta})}
\newcommand{\alphaftil}{(\widetilde{\alpha f})}

\newcommand{\eqtagwithsubscr}[1]{\stepcounter{equation}\tag*{(\theequation)$_{#1}$}}

\newcommand{\ainf}{$A_{\infty}$}
\newcommand{\dainf}{$dA_{\infty}$}
\newcommand{\dainfalg}{{\dainf}-algebra }
\newcommand{\dainfalgs}{{\dainf}-algebras }

\begin{document}
\title{DG-algebras and derived $A_{\infty}$-algebras}
\author{Steffen Sagave}
\address{Department of Mathematics, University of Oslo, Box 1053,  N-0316 Oslo, 
Norway}
\email{sagave@math.uio.no}
\date{\today}
\begin{abstract}
  A differential graded algebra can be viewed as an $A_{\infty}$-algebra.
  By a theorem of Kadeishvili, a dga over a field admits a
  quasi-isomorphism from a minimal $A_{\infty}$-algebra. We introduce the
  notion of a derived $A_{\infty}$-algebra and show that any dga $A$ over an
  arbitrary commutative ground ring $k$ is equivalent to a minimal derived
  $A_{\infty}$-algebra. Such a minimal derived $A_{\infty}$-algebra model
  for $A$ is a $k$-projective resolution of the homology algebra of $A$
  together with a family of maps satisfying appropriate relations.

  As in the case of $A_{\infty}$-algebras, it is possible to recover the
  dga up to quasi-isomorphism from a minimal derived $A_{\infty}$-algebra
  model. Hence the structure we are describing provides a complete
  description of the quasi-isomorphism type of the dga.
\end{abstract}
\maketitle

\section{Introduction}
An $A_{\infty}$-algebra over a commutative ring $k$ is a $\mZ$-graded
$k$-module $A$ with structure maps $m_j \colon A^{\tensor j} \to A[2-j]$
satisfying certain relations.  $A_{\infty}$-algebras were introduced by
Stasheff in the 1960's \cite{Stasheff_associativity}. We refer to Keller's
survey \cite{Keller_introduction} for background and for references to the
various applications $A_{\infty}$-algebras have found in different fields
of mathematics.

A differential graded algebra over $k$ is an $A_{\infty}$-algebra with
$m_1$ the differential, $m_2$ the multiplication, and $m_j = 0$ for $j \geq
3$. If $k$ is a field, a theorem of Kadeishvili \cite{Kadeishvili_homology}
asserts that every $k$-dga $A$ admits a quasi-isomorphism of
$A_{\infty}$-algebras from a minimal $A_{\infty}$-algebra, where minimal
means $m_1 = 0$. The underlying graded $k$-algebra of such a {\em minimal
  model} for $A$ is the homology algebra $H_*(A)$. The quasi-isomorphism
type of $A$ can be recovered from its minimal model. Therefore, minimal
$A_{\infty}$-algebras provide an alternative description for
quasi-isomorphism classes of dgas \cite[\S 3.3]{Keller_introduction}, and
the minimal $A_{\infty}$-structure on $H_*(A)$ specifies precisely the
additional information needed to reconstruct the quasi-isomorphism type of
a dga from its homology \cite[\S 2.1]{Keller_introduction}.

From now on let $k$ be a commutative ring. Then the conclusion of
Kadeishvili's theorem will in general not hold if the homology $H_*(A)$ is
not $k$-projective since for example a $k$-linear cycle selection morphism
may not exist.

The aim of the present paper is to generalize the notion of an
$A_{\infty}$-algebra to a context in which a dga $A$ admits a minimal model
without restrictions on $k$ or $H_*(A)$, and thereby to provide a different
description of the quasi-isomorphism type of $A$.

Our approach is motivated by the following two observations. On the one
hand, $A_{\infty}$-algebras are closely related to Hochschild cohomology.
For example, the $m_3$ of the minimal $A_{\infty}$-structure on $H_*(A)$ is
a cocycle in the Hochschild complex of $H_* (A)$ whose cohomology class
encodes relevant information about $A$; see
\cite{Benson_K_S_realizability}. On the other hand, Hochschild cohomology
is for many purposes the `wrong' cohomology theory if we apply it to
$k$-algebras which are not $k$-projective. In this case, one should rather
consider the derived Hochschild cohomology, which is also known as Shukla
cohomology. Following Baues and Pirashvili \cite{Baues_P_shukla}, the
latter cohomology theory can be defined as the Hochschild cohomology of a
$k$-projective dga resolving the algebra.

Our strategy for finding minimal models in the general context is therefore
to first take a $k$-projective resolution of $H_*(A)$ in the direction of a
new grading, and then to look for the analog to the minimal
$A_{\infty}$-structure on this bigraded $k$-module. To take the additional
grading into account, we introduce the notion of a {\em derived
  \ainf{}-algebra} ({\em \dainfalg} for short). A \dainfalg is a
$(\mN,\mZ)$-bigraded $k$-module $E$ with structure maps $m^E_{ij} \colon
E^{\tensor j} \to E[i,2-(i+j)]$ for $i\geq 0$ and $j\geq 1$ satisfying
appropriate relations. It is {\em minimal} if $m_{01}^E = 0$.  Both dgas
and $A_{\infty}$-algebras can be considered as \dainfalgs concentrated in
degree $0$ of the $\mN$-grading.  The notion of equivalence suitable for
our purposes is that of an $E_2$-equivalence of \dainf{}-algebras, which is
detected on the iterated homology with respect to $m_{01}^E$ and
$m_{11}^E$.

\begin{theorem}\label{thm:min_models}
  Let $A$ be a differential graded algebra over a commutative ring $k$.
  There exists a $k$-projective minimal \dainfalg $E$ together with an
  $E_2$-equivalence $E \to A$ of \dainf{}-algebras. This {\em minimal
    \dainfalg model} $E$ of $A$ is well defined up to $E_2$-equivalences
  between $k$-projective minimal \dainf{}-algebras. Particularly, the
  bigraded $k$-module $E$ together with the differential $m_{11}^E$ and the
  multiplication $m_{02}^E$ is a $k$-projective resolution of the graded
  $k$-algebra $H_*(A)$.
\end{theorem}

The first step in the proof of the theorem is to construct a \dainfalg $B$
with $m_{ij}^B =0$ if $i+j \geq 3$ together with a map $B \to A$ of
\dainfalgs such that the induced map $H_*(B,m_{01}^B) \to H_*(A)$ is a
$k$-projective resolution. To find such a $B$, we employ the cofibrant
replacement in a resolution model category structure on the category of
simplicial dgas. The second step is to replace $B$ by a minimal
\dainf{}-algebra, which is now possible since $H_*(B,m_{01}^B)$ is
$k$-projective.

If $E$ is a minimal \dainfalg model of the dga $A$, the sum $m_{03}^E +
m_{12}^E + m_{21}^E$ is a cocycle in a complex computing the derived
Hochschild cohomology of $H_*(A)$. Its cohomology class does not depend on
the choice of the minimal model. This derived Hochschild cohomology class
$\gamma_A \in \dHH^{3,-1}(H_*(A))$ associated with $A$ generalizes the
Hochschild cohomology class of a dga over a field studied by Benson,
Krause, and Schwede in \cite{Benson_K_S_realizability}. For example,
$\gamma_A$ determines all triple matric Massey products in $H_*(A)$.

Let $k = \mZ/p^2$ with $p$ an odd prime, and let $A$ be the dga which has a
copy of $k$ in degrees $0$ and $1$ and multiplication by $p$ as the
differential $A_1 \to A_0$. Then $H_*(A)$ is an exterior algebra over
$\mF_p$ on a class in degree one. In Example \ref{ex:minimalmodel}, we
describe a minimal \dainfalg model of $A$. The derived Hochschild
cohomology class $\gamma_A$ defined by means of this minimal model is
non-zero.
 
The information encoded in a minimal \dainfalg model is not restricted to
the resolution of $H_*(A)$ and the triple Massey products. In fact, we can
recover the dga from a minimal \dainfalg model:

\begin{theorem}\label{thm:antimin_models}
For a \dainfalg $E$, there is an associated differential graded algebra 
$\Tot \enrhom_E(E,E)$. If $E$ is a minimal \dainfalg model of a dga $A$,
the dga $\Tot \enrhom_E(E,E)$ is quasi-isomorphic to $A$. 
\end{theorem}

Therefore, we have given an answer to the question of which additional
structure is needed to reconstruct the quasi-isomorphism type of a dga over
a commutative ground ring from its homology algebra: A $k$-projective
resolution of its homology equipped with a minimal \dainf{}-algebra
structure provides the necessary information.

A different approach to describe quasi-isomorphism types of dgas is to use
Postnikov systems and $k$-invariants. This approach is used by Dugger and
Shipley in \cite[\S 4]{Dugger_S_topological} for their study of topological
equivalences of dgas. One advantage of the minimal \dainfalg structures
introduced in the present paper is that they also exist for dgas with
homology in negative degrees.

One may apply our results to $\Ext$-algebras. If $M$ is a $k$-module, the
Yoneda-algebra $\Ext_k^*(M,M)$ is the homology of the endomorphism dga
$\Hom_k(P,P)$ of a $k$-projective resolution $P$ of $M$. Hence a minimal
\dainfalg model for $\Hom_k(P,P)$ provides a $k$-projective resolution of
$\Ext^*_k(M,M)$ with structure maps $m_{ij}$. This structure is well
defined up to $E_2$-equivalence, and it reflects the quasi-isomorphism type
of $\Hom_k(P,P)$. We calculate it for $\Ext_{\mZ}^*(\mZ/p, \mZ/p)$. This
generalizes the existence of a minimal $A_{\infty}$-structure on the
extension algebra of a module over an algebra over a field \cite[\S
3.3]{Keller_introduction}.

\subsection*{Organization}
In the second section, we define \dainfalgs and explain how they generalize
$A_{\infty}$-algebras. We introduce bidgas and twisted dgas and define the
notion of $E_2$-equivalence. In the third section, we explain how dgas can
be resolved by bidgas. This makes use of Bousfield's theory of resolution
model categories. The fourth section features the proof of Theorem
\ref{thm:min_models}.  In Section \ref{sec:ex_appl}, we define the derived
Hochschild cohomology class associated with a dga and describe the examples
mentioned in the Introduction.  In the sixth section, we introduce modules
over \dainfalgs and use the enrichment of module categories in twisted
chain complexes to prove Theorem \ref{thm:antimin_models}.  The last
section contains four lemmas needed in Section \ref{sec:minimal}.

\subsection*{Notation and conventions}
Throughout the paper, $k$ is a commutative ground ring, and all tensor
products are taken over $k$. We consider $(\mN, \mZ)$-bigraded $k$-modules
$E = (E_{st})_{s \in \mN, t \in \mZ}$.  Their shift is defined by
$E[i,j]_{s t} = E_{s - i, t - j}$. We sometimes refer to the $\mN$-grading
as the horizontal direction and to the $\mZ$-grading as the vertical
direction.

The tensor product of two bigraded $k$-modules $E$ and $F$ is the bigraded
$k$-module $E \tensor F$ with
\begin{equation} \label{equ:tensor_bigraded}
(E \tensor F)_{uv} = \bigoplus_{i+p=u,\, j+q=v} E_{ij} \tensor
E_{pq}.
\end{equation}
We follow the Kozul sign convention. For maps $f \colon W \to X[i,j], g
\colon Y \to Z[p,q]$ and elements $w \in W_{st}, y \in Y_{uv}$ this means
$(f \tensor g) (w \tensor y) = (-1)^{ \langle w,g \rangle } f(w) \tensor
g(y)$ with $\langle w,g \rangle = s p + t q$. Similarly, we get $(f \tensor
g) (f' \tensor g') = (-1)^{ \langle f',g \rangle } ff' \tensor gg' $ for
tensor products of composable maps. Tensor product, shift, and sign
convention for $k$-modules with a single $\mZ$-grading are given
accordingly.

We assume familiarity with the basics about $A_{\infty}$-algebras, as for
example described in Section 3 of Keller's survey
\cite{Keller_introduction}.  In our sign convention for
$A_{\infty}$-algebras, we follow Lef\`evre \cite{Lefevere_sur} and ask the
structure maps $m_j$ of an $A_{\infty}$-algebra to satisfy $\sum
(-1)^{rq+t} m_j(\eins^{\tensor r} \tensor m_{q} \tensor \eins^{\tensor t})
= 0$. Remark \ref{rem:signs_explanation} explains our choice of signs.

\section{\texorpdfstring{\dainfalgs and twisted chain complexes}{Derived
    A-infinity algebras and twisted chain complexes}}
Our main object of study is the following generalization of an
$A_{\infty}$-algebra whose definition was already outlined in the
Introduction.

\begin{definition} \label{def:dainf_obj} A {\em derived \ainf{}-algebra}
  ({\em \dainf{}-algebra} for short) is a $(\mN,\mZ)$-bigraded $k$-module
  $E$ with $k$-module maps $m_{ij} \colon E^{\tensor j} \to E[i,2-(i+j)]$
  and a unit map $\eta_E \colon k \to E$ for $i\geq 0, j \geq 1$ satisfying
\[
\sum_{\substack{i + p =u, \, j + q = v + 1\\ r + 1 + t = j \\ 
j,q \geq 1, \, i,r,p,t \geq 0}} 
(-1)^{rq+t+pj} m_{i j}^E(\eins^{\tensor r} \tensor m_{p q}^E \tensor \eins^{\tensor t}) = 0
\eqtagwithsubscr{uv}\label{equ:obj_rel}
\]
for all $u \geq 0, v \geq 1$ and the unit condition:
\begin{equation}\label{equ:obj_unitcondition}
\begin{split}
m_{01}^E \eta_E = 0,\quad  m_{02}^E(\eta_E \tensor \eins) = \eins = m_{02}^E(\eins \tensor
\eta_E), \quad \text{ and} \\  m_{ij}^E(\eins^{\tensor r-1} \tensor \eta_E
\tensor \eins^{\tensor j-r}) = 0
\text{ if } i+j \geq 3 \text{ with } 1 \leq r \leq j
\end{split}
\end{equation} 
\end{definition}

\begin{example}
  A derived \ainf{}-algebra concentrated in horizontal degree $0$ is the
  same as a (strictly unital) $A_{\infty}$-algebra with the $m_{0j}$ as
  structure maps: the equations \ref{equ:obj_rel} for $u=0$ are the
  $A_{\infty}$-relations.  Since a $k$-dga $A$ is a special instance of an
  $A_{\infty}$-algebra, it is a \dainfalg concentrated in horizontal degree
  $0$ for which only $m_{01}^A$ and $m_{02}^A$ may be non-zero.
\end{example}

\begin{definition} \label{def:dainf_maps} A map of \dainfalgs from $E$ to
  $F$ consists of a family of $k$-module maps $f_{p q} \colon E^{\tensor q}
  \to F[p, 1-(p+q)]$ with $p \geq 0$ and $q \geq 1$ satisfying
\[\begin{split} &\sum_{\substack{i+p=v, \, j+q = v+1 \\ r+1+t=j  \\ j,q \geq 1, i,r,p,t \geq 0}}
 (-1)^{rq+t+pj}  f_{i j}(\eins^{\tensor r} \tensor m^E_{p q} \tensor \eins^{\tensor t})\\
=& \sum_{\substack{ i + p_1 + \dots + p_j  = u \\
q_1 + \dots + q_j = v \\ 1 \leq j \leq v, \, p_s \geq 0, \, q_s \geq 1}}
(-1)^{\sigma} m^F_{i j}(f_{p_1  q_1} \tensor \dots \tensor f_{p_j  q_j})
\end{split}\eqtagwithsubscr{uv}\label{equ:mor_rel}\]
for $u \geq 0, v \geq 1$ and the unit condition
\begin{equation}\label{equ:mor_unitcondition}
f_{01} \eta_E = \eta_F \text{ and }
f_{pq}(\eins^{\tensor r-1} \tensor \eta_E \tensor \eins^{\tensor q-r}) = 0
\text{ if } p+q \geq 2  \text{ with } 1 \leq r \leq q.
\end{equation}  
The $\sigma$ governing the sign in \ref{equ:mor_rel} is
\begin{equation} \label{equ:map_sign}
\sigma = u+ \sum_{w=1}^{j-1} \left(j p_w + w (q_{j-w} - p_w) + q_{j-w}
\left(\sum_{s=j-w+1}^j p_s+q_s\right)\right). 
\end{equation} 
\end{definition}

\begin{example}
  Maps between \dainfalgs concentrated in horizontal degree $0$ are maps of
  $A_{\infty}$-algebras.
\end{example}

In Section \ref{sec:minimal}, we will give an equivalent but more economic
description of \dainfalgs and their maps in terms of a structure on the
reduced tensor algebra of $E[0,1]$. Amongst others, we employ it to define
the composition of \dainfalg maps and the notion of a module over a
\dainf{}-algebra, and it will also help to explain the signs. We use the
less streamlined Definition \ref{def:dainf_obj} here since both our
examples and our notion of equivalence fit in more naturally.

We single out an interesting special case of a \dainf{}-algebra:

\begin{definition} \label{def:bidgas} A {\em bidga} is a \dainfalg $B$ with
  $m_{ij}^B = 0$ if $i+j \geq 3$. A map $f\colon A \to B$ of bidgas is a
  map of \dainfalgs with $f_{pq} = 0$ for $p+q \geq 2$. We write $\bidga_k$
  for the resulting category.
\end{definition}

Of course, dgas give rise to bidgas concentrated in horizontal degree $0$.

\begin{remark}\label{rem:bidga_monoids}
  The category $\bidga_k$ admits a different description. Let $\biChxes_k$
  be the category of $(\mN,\mZ)$-graded bicomplexes of $k$-modules. Objects
  are bigraded $k$-modules $E$ with differentials $d_0 \colon E \to E[0,1]$
  and $d_1 \colon E \to E[1,0]$ satisfying $d_0 d_0 = 0$, $d_1 d_1 = 0$,
  and $d_0 d_1 = d_1 d_0$. The tensor product of bigraded $k$-modules
  induces a symmetric monoidal product on $\biChxes_k$. A monoid in
  $\biChxes_k$ is a bicomplex $B$ with a multiplication $\mu \colon B
  \tensor B \to B$ satisfying associativity, two Leibniz rules, and a unit
  condition. Setting $m_{01}^B = d_0$, $m_{11}^B = d_1$ and $m_{02}^B =
  \mu$, we observe that the six non-trivial formulas \ref{equ:obj_rel} for
  $u+v \leq 3$ are the associativity law, the two Leibniz rules, and the
  three identities for the differentials. After comparing the unit
  conditions, we see that $\bidga_k$ is the category of monoids in
  $\biChxes_k$. This is similar to the characterization of dgas as the
  monoids in the category of chain complexes with respect to the tensor
  product of chain complexes.
\end{remark}

\begin{definition} \label{def:tdgas} A {\em twisted dga} ($\tdga$ for
  short) is a \dainfalg $T$ for which only $m_{02}^T$ and $m_{i1}^T$ with
  $i \geq 0$ may be non-zero.  A map $f\colon T \to S$ of tdgas is a map of
  \dainfalgs with $f_{pq} = 0$ if $q \geq 2$.  We write $\tdga_k$ for the
  resulting category.
\end{definition}

Every bidga is a tdga. The name {\em twisted dga} is chosen in view of

\begin{definition}\cite{Gelfand_homological} \label{def:tch}
  A {\em twisted chain complex} $E$ is an $(\mN,\mZ)$-graded $k$-module
  with differentials $d_i^E \colon E \to E[i,1-i]$ for $i \geq 0$
  satisfying
  \begin{equation} \eqtagwithsubscr{u} \label{equ:tch_obj_rel} \sum_{i+p=u}
    (-1)^i d_i d_p = 0
  \end{equation}
  for $u \geq 0$. A map of twisted complexes $E \to F$ is a family of maps
  $f_i \colon E \to F[i,-i]$ satisfying
\begin{equation} \eqtagwithsubscr{u} \label{equ:tch_map_rel}
\sum_{i+p=u} (-1)^i f_i d^E_p = \sum_{i+p=u} d^F_i f_p.
\end{equation}
The composition of two maps $f \colon E \to F$ and $g \colon F \to G$ is
defined by $(gf)_u = \sum_{i+p=u} g_i f_p$. We write $\tChxes_k$ for the
resulting category.
\end{definition}

A class of examples arises from the fact that $\biChxes_k$ is a subcategory
of $\tChxes_k$. The signs in Definition \ref{def:tch} differ from those in
\cite{Gelfand_homological}. We explain our choice in Remark
\ref{rem:signs_explanation}. Some authors use the term `multicomplex'
instead of `twisted chain complex'.

\begin{remark}
  Similarly as in Remark \ref{rem:bidga_monoids}, the tensor product of
  graded modules defines a symmetric monoidal product on $\tChxes_k$.
  Monoids in $\tChxes_k$ are tgas. The resulting composition tdga maps is
  of course a special case of the composition of \dainfalg maps to be given
  in Definition \ref{def:map_comp}.
\end{remark}

Let $E$ be a \dainf{}-algebra. Identifying \ref{equ:obj_rel} for $v=1$ with
\ref{equ:tch_obj_rel}, we see that the $m^E_{i1}$ with $i \geq 0$ specify
the {\em underlying twisted chain complex} of $E$.  For a map $f \colon E
\to F$ of \dainf{}-algebras, the $f_{i1}$ form a map of the underlying
twisted chain complexes. We thus occasionally write $d_i$ for $m_{i1}^E$
when we refer to the underlying twisted chain complex of a \dainfalg $E$.

\begin{remark}
  The underlying twisted chain complex is one instance in which twisted
  chain complexes behave to \dainfalgs as unbounded chain complexes behave
  to \ainf{}-algebras.  This analogy turns out to be fruitful, and we
  frequently exploit it in Sections \ref{sec:minimal} and \ref{sec:anti}.
\end{remark}

Next we study the passage from twisted chain complexes to chain complexes.
There is a total complex functor
\[ \Tot \colon \tChxes_k \to \Chxes_k, \qquad \Tot_n X = \bigoplus_{s+t=n}
X_{st}, \] where the differential $d \colon \Tot_n X \to \Tot_{n-1} X$ is
$\sum_{i\geq 0} d_i^X$. A map of twisted complexes $f \colon X \to Y$
induces a map $\sum_{i\geq 0} f_i$ of total complexes. This is well defined
since only finitely many $d_i^X$ or $f_i$ leave a fixed $X_{st}$.

If we consider the total complex as a twisted chain complex concentrated in
horizontal degree $0$, there is a map $\rho_X \colon X \to \Tot X$ of
twisted chain complexes. Its component $(\rho_X)_i \colon X \to (\Tot
X)[i,-i]$ is the inclusion of $X_{i*}$ into $(\Tot X)[i,-i]_{i*}$. We can
interpret $\rho$ as a natural transformation from the identity on
$\tChxes_k$ to $\Tot$. The presence of this natural transformation is one
advantage of the category of twisted chain complexes over the category of
bicomplexes.

\begin{lemma}\label{lem:Tot_mult}
  The functor $\Tot \colon \tChxes_k \to \Chxes_k$ is strongly monoidal,
  and $\rho$ is a strong monoidal transformation. In particular, if $E$ is
  a tdga, then $\Tot E$ is a dga, and $\rho_E \colon E \to \Tot E$ is a map
  of tdgas.
\end{lemma}

The total complex has a filtration defined by $F^p(\Tot X)_n = \bigoplus_{s
  \leq p} X_{s,n-s}$. The $E_2$-term of the resulting spectral sequence is
$E^2_{pq} \iso H^h_p H^v_q(X)$. Here we write $H^v_*(E)$ for the `vertical'
homology of $E$ with respect to the differential $d_0$. By
\ref{equ:tch_obj_rel} for $u=2$, $d_1$ induces a differential on
$H^v_*(E)$, and we write $H_*^h(H_*^v(E))$ for the resulting `horizontal'
homology group.

Maps of twisted complexes induce maps of spectral sequences: given $f
\colon E \to F$ in $\tChxes_k$, $f_{0}$ is a $d_0$-chain map on $E \to F$,
and formula \ref{equ:tch_map_rel} for $u = 1$ ensures that $H_*^v(f_{0})$
induces chain map with respect to the $d_1$-differential $H_*^v(d_1)$.
Hence we obtain $H_*^h(H_*^v(f_{0}))$ on the $E_2$-term, and so on.

\begin{definition} \label{def:E2equiv} A map $f \colon E \to F$ of twisted
  complexes is an {\em $E_2$-equivalence} if $H^h_s(H_t^v(f_{0}))$ is an
  isomorphism for $s \in \mN, t \in \mZ$. A map of bicomplexes, bidgas,
  twisted dgas or derived \ainf{}-algebras is an $E_2$-equivalence if the
  underlying map of twisted chain complexes is.
\end{definition}
Accordingly, $f$ is an $E_1$-equivalence if it induces an isomorphism on
$H^v_*$, and every $E_1$-equivalence is an $E_2$-equivalence.

The following lemma is immediate.
\begin{lemma}\label{lem:Tot_E2equiv}
  If $X$ is a twisted chain complex with $E_2$-homology concentrated in
  horizontal degree $0$, then $\rho_X \colon X \to \Tot X$ is an
  $E_2$-equivalence.
\end{lemma}

\section{Resolving dgas by bidgas}
The aim of this section is to construct resolutions of dgas that induce
resolutions of their graded homology algebras. We saw that a dga can be
considered as a bidga concentrated in horizontal degree $0$. Particularly,
a graded $k$-algebra is a bidga concentrated in horizontal degree $0$ with
both differentials trivial.

\begin{definition} \label{def:alg_res} A {\em termwise $k$-projective
    resolution} of a graded algebra $\Lambda$ is a termwise $k$-projective
  bidga $E$ with trivial vertical differential together with an
  $E_2$-equivalence $E \to \Lambda$ of bidgas.
\end{definition}

Here we think of a bidga with trivial vertical differential as a
differential graded algebra in {\em graded} algebras. If $E \to \Lambda$ is
a termwise $k$-projective resolution, $E_{*, 0}$ is a dga which is
$k$-projective and whose homology is isomorphic to $\Lambda_0$ concentrated
in degree $0$. In other words, a termwise $k$-projective resolution is the
graded analog of the resolution of a $k$-algebra by a degreewise
$k$-projective dga.

If the graded algebra is the homology of a differential graded algebra, we
can ask even more:

\begin{definition} \label{def:dga_res} Let $A$ be a $k$-dga. A {\em
    $k$-projective $E_1$-resolution} of $A$ is a bidga $B$ together with an
  $E_2$-equivalence $B \to A$ of bidgas such that $H^v_{st}(B)$ is
  $k$-projective and the map $k \to H^v_{00}(B)$ induced by the unit $k \to
  B$ splits as a $k$-module map.
\end{definition}

A termwise $k$-projective $E_1$-resolution of a dga $A$ is a termwise
$k$-projective resolution of the graded algebra $H^v_*(A)$ which is induced
from a resolution of the dga and satisfies an additional unit condition.

\begin{example} \label{ex:E1res} Let $p$ be an odd prime and $k =\mZ/p^2$.
  The $k$-dga $A = \Lambda^*_{\mZ/p^2}(w)$ with $|w|=1$ and $d(w) = p$ is
  non-formal and has homology $\Lambda^*_{\mF_p}(\ovl{w})$ with
  $|\ovl{w}|=1$. It is also considered in \cite[\S 5.2]{Baues_P_shukla}.
  We will give a $k$-projective $E_1$-resolution of $A$.

  Consider the $(\mN,\mZ)$-bigraded $k$-algebra
  \[ V = k[a,b,u,v]/(a^2, b^2, au-bv, auv, buv, u^2, v^2), \] with
  $|a|=|b|=(0,0)$ and $|u|=|v|=(0,1)$. We set $d^V_0 (u) = a$ and $d^V_0(v)
  = b$. This extends uniquely to a differential on $V$. We think of $V$ as
  a dga concentrated in horizontal degree $0$. Its homology is
  $\Lambda_k^*(\ovl{z})$ with $\ovl{z}$ is represented by $au$. Let $H$ be
  the $(\mN, \mZ)$-bigraded $k$-algebra
  \[ 
  \Lambda_k^*(x) \tensor \Gamma_k^*(y) \quad \textrm{with} \quad |x|=(1,0),
  |y|=(2,0),
  \] 
  concentrated in vertical degree $0$. We define $B = V \tensor H$, so that
  $B_{st} = V_{s0} \tensor H_{0t}$. The multiplications on $V$ and $H$ turn
  $B$ into a bigraded $k$-algebra.  Since $d^H_0 =0$ for degree reasons,
  the $d^V_0$ on $V$ induces a vertical differential $d^B_0$ satisfying the
  Leibniz rule. We define $d^B_1(xy_i) = (p-pa+b)y_i$ and $d^B_1(y_i)=
  (p-b-ab)x y_{i-1}$, where $y_i$ is the divided power algebra generator in
  degree $2i$. The $d^B_1$ is linear with respect to $a, b, u,$ and $v$,
  and $d^B_0$ and $d^B_1$ commute.

  It is easy to see $H^v_*(B) = \Lambda_k^*(\ovl{x}, \ovl{z}) \tensor
  \Gamma_k^*(\ovl{y})$ as well as $d_1(\ovl{y}_i) = p \, \ovl{x}\,
  \ovl{y}_{i-1}$ and $d_1(\ovl{x} \,\ovl{y}_i) = p \ovl{y}_{i}$.
  Hence $B$ is a bidga for which $H^v_*(B)$ is bidegreewise $k$-projective.

  A map of bidgas $\alpha \colon B \to A$ is defined by setting
  \[ 
  \alpha(y_0) = 1, \alpha(a) = p, \alpha(b) = -p, \alpha(u)=w, \textrm{ and
  } \alpha(v)=-w.
  \]
  This is multiplicative and compatible with the relations among $a,b,u,$
  and $v$. Since $\alpha(d_1^B(x))=0$, the $\alpha$ is also compatible with
  $d^B_1$. Therefore, it is a bidga map.  On vertical homology $\alpha$
  induces $H^v_*(\alpha) \colon H^v_*(B) \to H^v_*(A)$ sending $\ovl{y}_0$
  to $1$ and $\ovl{z}$ to $\ovl{w}$. This is easily seen to be a resolution of
  $H^v_*(A)$.
\end{example}

Returning to the general case, a map of $k$-projective $E_1$-resolutions
from $B \to A$ to $B' \to A$ is a bidga map $f \colon B \to B'$ making the
obvious triangle commutative. It is defined to be an $E_2$-equivalence if
$f$ is one.

\begin{theorem} \label{thm:exis_uni_res} Every $k$-dga $A$ admits a
  $k$-projective $E_1$-resolution $B \to A$.  Two such resolutions can be
  related by a zig-zag of $E_2$-equivalences between $k$-projective
  $E_1$-resolutions.
\end{theorem}

The proof of the theorem occupies the rest of the section. It uses
Quillen's language of model categories \cite{Quillen_homotopical,
  Hovey_model}. The point is to use the cofibrant replacement in an
appropriate model structure on the category $\sdga_k$ of simplicial dgas.

We start by briefly recalling Bousfield's general setup for resolution
model structures \cite{Bousfield_resolution}, or rather the dual version of
\cite{Jardine_E2}. Let $\cC$ be a pointed model category. Let $\cP$ be a
class of group-objects in $\Ho(\cC)$. A morphism $p \colon X \to Y$ in
$\Ho(\cC)$ is a {\em $\cP$-epi} if $p_* \colon [P,X]_n \to [P,Y]_n$ is onto
for every $P \in \cP$ and $n\geq 0$, where $[X,Y]_n = \Ho(\cC)(\Sigma^n
X,Y)$. An object $A \in \Ho(\cC)$ is {\em $\cP$-projective} if $p_* \colon
[A,X]_n \to [A,Y]_n$ is onto for all $\cP$-epis $p$ and $n\geq 0$. Maps in
$\cC$ (resp.  objects in $\cC$) are $\cP$-epis (resp. $\cP$-projective) if
they are in $\Ho(\cC)$.  A map in $\cC$ is a {\em $\cP$-projective
  cofibration} if it has the left lifting property with respect to all
$\cP$-epi fibrations. A map $X \to Y$ in $\cC$ is {\em $\cP$-free} if it is
a composition of an inclusion $X \to X \coprod P$ with $P$ cofibrant and
$\cP$-projective and an acyclic cofibration $X \coprod P \to Y$.  We say
that {\em $\Ho(\cC)$ has enough $\cP$-projectives} if every object $X \in
\Ho(\cC)$ admits a $\cP$-epi $Y \to X$ with $\cP$-projective source. In
this case, we call $\cP$ a class of projective models in $\Ho(\cC)$.

The next definition makes use of the standard model structure on the
category of simplicial groups (\cite[II.3.7]{Quillen_homotopical}), the
Reedy model structure on $s\cC$ \cite[VII.2.12]{Goerss_J_simplicial}, and
the notion of the latching object $L_n X$ and its structure map $L_n X \to
X_n$ for $X \in s\cC$ \cite[VII.1.5]{Goerss_J_simplicial}.

\begin{definition} \label{def:Pmodel_str} Let $f \colon X \to Y$ be a map
  in $s\cC$.
\begin{enumerate}[(i)]
\item $f$ is a {\em $\cP$-equivalence} if $f_* \colon [P,X]_n \to [P,Y]_n$
  is a weak equivalence in simplicial groups for all $P \in \cP$ and $n
  \geq 0$.
\item $f$ is a {\em $\cP$-fibration} if it is a Reedy-fibration and $f_*
  \colon [P,X]_n \to [P,Y]_n$ is a fibration of simplicial groups for all
  $P \in \cP$ and $n \geq 0$.
\item $f$ is a {\em $\cP$-cofibration} if the induced maps $X_n
  \coprod_{L_nX} L_nY \to Y_n$ are $\cP$-projective cofibrations for all
  $n\geq 0$.
\end{enumerate}
\end{definition}

\begin{theorem} \label{thm:res_str} \cite[Theorem
  3.3]{Bousfield_resolution} Let $\cC$ be a right proper pointed model
  category and let $\cP$ be a class of projective models for $Ho(\cC)$.
  With the above $\cP$-equivalences, $\cP$-fibrations, and
  $\cP$-cofibrations, $s\cC$ becomes a right proper model category. It is
  simplicial with respect to the external simplicial structure.
\end{theorem}

We frequently use the following feature of this model structure:

\begin{lemma} \label{lem:cof_term} A $\cP$-cofibration in $s\cC$ is
  termwise a $\cP$-projective cofibration in $\cC$.
\end{lemma}
\begin{proof}
  This is a consequence of \cite[Lemma 15.3.9]{Hirschhorn_model}, compare
  \cite[Lemma 5.3]{Bousfield_resolution}.
\end{proof}

Our main example is the category $\Chxes_k$ with the projective model
structure \cite[Theorem~2.3.11]{Hovey_model}. Since $\Chxes_k$ is stable,
all objects in $\Ho(\Chxes_k)$ are cogroup objects. Let $\cP$ be the set of
objects $\{ S^n(k) | n \in \mZ\}$ in $\Ho(\Chxes_k)=\cD(k)$, where
$S^n(k)_i = k$ if $i=n$ and $0$ otherwise. The $\cP$-epis are the maps
which induce surjections in homology.

\begin{lemma} \label{lem:eno_Pproj} With this $\cP$, the category
  $\Ho(\Chxes_k)$ has enough $\cP$-projectives. An object in
  $\Ho(\Chxes_k)$ is $\cP$-projective if and only if it has degreewise
  $k$-projective homology.
\end{lemma}
\begin{proof}
  Since every object is the codomain of a $\cP$-epi mapping out of a sum of
  objects in $\cP$, there are enough $\cP$-projectives.  Applying this to a
  $\cP$-projective object $X$ shows that $X$ is a retract of a sum of
  objects in $\cP$. Hence $X$ has degreewise $k$-projective homology.  On
  the other hand, an object with degreewise $k$-projective homology is a
  retract of a sum of objects of $\cP$, and therefore $\cP$-projective.
\end{proof}

All objects in $\Chxes_k$ are fibrant, so $\Chxes_k$ is right proper.
Theorem \ref{thm:res_str} yields the {\em $\cP$-model structure} on
$\sChxes_k$.

Analogous to the classical case of simplicial abelian groups and simplicial
rings, simplicial dgas can be considered as the monoids in the category
$\sChxes_k$ of simplicial chain complexes with respect to the termwise
tensor product.

\begin{proposition} \label{prop:mstr_sdga} The forgetful functor $U \colon
  \sdga_k \to \sChxes_k$ creates a cofibrantly generated model structure on
  $\sdga_k$ in which a map $f$ is a $\cP$-fibration or a $\cP$-equivalence
  if $U(f)$ is. Cofibrant objects in $\sdga_k$ are termwise
  $\cP$-projective cofibrant in $\Chxes_k$. The unit map of a cofibrant
  replacement in $\sdga_k$ splits in homology.
\end{proposition}

\begin{proof}
  First notice that the $\cP$-model structure on $\sChxes_k$ is cofibrantly
  generated. To obtain the generating cofibrations $I$ and generating
  acyclic cofibrations $J$, one has to add the generating acyclic
  cofibrations of the Reedy model structure \cite[Proposition
  15.6.24]{Hirschhorn_model} to the two sets of maps described in
  \cite[Lemma 2.7]{Jardine_E2}.

  The projective model structure on $\Chxes_k$ creates a model structure on
  $\dga_k$ \cite[\S 5]{Schwede_S_algebras}, which gives rise to a Reedy
  model structure on $\sdga_k$. As the matching object functor is defined
  as a limit and the forgetful functor $U$ commutes with limits, a map $f$
  in $\sdga_k$ is a Reedy fibration or a Reedy weak equivalence if and only
  if $U(f)$ is.
  
  Since all objects in simplicial abelian groups are fibrant, the
  $\cP$-fibrant objects in $\sChxes_k$ coincide with the Reedy fibrant
  objects. Therefore, the Reedy fibrant replacement functor in $\sdga_k$ is
  a fibrant replacement functor for the $\cP$-model structure to be
  constructed.
  
  The cotensor of the external simplicial structure on $\sChxes_k$ is
  defined as a limit \cite[p. 371]{Goerss_J_simplicial} and therefore
  extends to a functor
  \[\homsp_{\sdga_k}\colon  \cS^{\op} \times \sdga_k \to  \sdga_k.\]  
  As the $\cP$-model structure on $\sChxes_k$ is simplicial,
  $\homsp_{\sdga_k}(-,Y)$ maps weak equivalences to $\cP$-equivalences and
  cofibrations to $\cP$-fibrations provided $Y$ is $\cP$-fibrant. Hence the
  simplicial path object for a fibrant object qualifies as a path object
  for the $\cP$-model structure on $\sdga_k$ we are heading for.
  
  Since the $\cP$-model structure on $\sChxes_k$ is cofibrantly generated,
  the verification of the axioms for the $\cP$-model structure on $\sdga_k$
  now follows from applying \cite[Lemma 2.3(2)]{Schwede_S_algebras} to the
  left adjoint of $U$. As in Quillen's original argument
  \cite[II.4.9]{Quillen_homotopical}, one can reduce the assumptions to
  presence of a path object and the fibrant replacement functor as
  established above.
  
  The generating cofibrations $T(I)$ of the $\cP$-model structure on
  $\sdga_k$ are obtained by applying the free associative algebra functor
  $T$ to $I$.  If $X$ is $\cP$-projective and cofibrant, $ - \tensor X$
  preserves $\cP$-free maps. Cobase change also preserves $\cP$-free maps.
  The maps in $I$ are either termwise acyclic cofibrations in the
  projective model structure on $\Chxes_k$, or termwise of the form $X \to
  X \coprod P$ with $P \in \cP$. By the careful analysis of the free
  associative algebra functor carried out in \cite[Lemma
  6.2]{Schwede_S_algebras}, it follows that for a termwise $\cP$-projective
  cofibrant $X$ in $\sdga_k$, the map $X \to Y$ obtained from attaching a
  generating cofibration $f \in T(I)$ to $X$ is is isomorphic to the
  inclusion of $X$ into the colimit of a countable sequence of termwise
  $\cP$-free maps: inspecting the proof the second statement of the cited
  lemma, we see that it is enough to provide the last four statements and
  Lemma \ref{lem:ppa_projcof}.

  The initial object of $\sdga_k$ is itself termwise $\cP$-projective.
  Hence the map from the initial object $k$ to a cofibrant replacement
  constructed by means of the small object argument applied to $T(I)$ is
  termwise a transfinite composition of $\cP$-free maps. Therefore, it
  splits in homology. In particular, any cofibrant replacement is termwise
  $\cP$-projective cofibrant. By the retract argument, all cofibrant
  objects are termwise $\cP$-projective cofibrant.
\end{proof}

Let $f \colon X \to Y$ and $g\colon X'\to Y'$ be maps in $\Chxes_k$ or
$\sChxes_k$. We define the pushout product map $f \Box g$ to be the
induced map $Y \tensor X' \coprod_{X \tensor X'} X \tensor Y' \to Y \tensor
Y'$.

\begin{lemma} \label{lem:ppa_projcof} Let $K$ be the class of termwise
  acyclic cofibrations in $\sChxes_k$, and let $L$ be the class of maps in
  $\sChxes_k$ which are termwise of the form $X \to X \coprod P$ with $P$ a
  $\cP$-projective cofibrant object. Then $(K \cup L) \Box (K \cup L)
  \subset (K \cup L)$.
 \end{lemma}
\begin{proof} 
  The pushout product axiom \cite[Definition
  3.1]{Schwede_Shipley_monoidal} holds in $\Chxes_k$. It states that
  if $f$ and $g$ are cofibrations in $\Chxes_k$, then the pushout
  product map $f\Box g$ is a cofibration which is acyclic if $f$ or
  $g$ is. The axiom immediately implies $K \Box K \subset K$. Since
  objects in $\cP$ are cofibrant, it also shows that for $P \in \cP$
  and $f \in K$, the map $P \tensor f$ is again in $K$. This argument
  implies $K \Box L \cup L \Box K \subset K$. If both maps are in $L$,
  it is enough to show that for two $\cP$-projective cofibrant objects
  $P, P' \in \cP$, their product $P \tensor P'$ is $\cP$-projective
  cofibrant. Cofibrancy is clear. Since both have $k$-projective
  homology, the $E_2$-term of the bicomplex spectral sequence
  computing the homology of $P \tensor P'$ consists only of a single
  non-trivial line in which all entries are $k$-projective. By Lemma
  \ref{lem:eno_Pproj}, $P\tensor P'$ is $\cP$-cofibrant.
\end{proof}

We need an instance of the Dold-Kan correspondence to go back and forth
between simplicial chain complexes and bicomplexes. The version in
\cite[III Theorem 2.5]{Goerss_J_simplicial} is sufficiently general to
apply to our context.

The associated chain complex and normalized chain complex are functors 
\[ C \colon \sChxes_k \to \biChxes_k \quad \text{ and } \quad N \colon
\sChxes_k \to \biChxes_k.\] The first is defined by $C(X)_{s*} = X_s$ with
horizontal differential given by the alternating sum of the simplicial face
maps, and $N(X)$ is the quotient of $C(X)$ by the subobject generated by
the image of the degeneracies.  The functor $N$ is an equivalence of
categories. Its inverse $\Gamma \colon \biChxes_k \to \sChxes_k$ can be
defined by
\[\Gamma(E)_n = \bigoplus_{[n] \twoheadrightarrow [p]} E_{p*}\]
with the simplicial structure maps described on \cite[p.
148]{Goerss_J_simplicial}.

The monoidal properties of these functors are as in the classical case
discussed in \cite[\S 2.3,\S 2.4]{Schwede_Shipley_monoidal}.  We have a
shuffle map $\nabla$ and the Alexander-Whitney map $AW$,
\[ \nabla \colon CX \tensor CY \to C(X \tensor Y) \quad \text{ and } \quad
AW \colon C(X \tensor Y) \to CX \tensor CY,  \]
$\nabla$ is a lax monoidal
transformation, $AW$ is a lax comonoidal transformation, and both maps
preserve the subcomplexes of degenerate simplicies and induce
\[ \nabla \colon NX \tensor NY \to N(X \tensor Y) \quad \text{ and } \quad
AW \colon N(X \tensor Y) \to NX \tensor NY. \]

For a simplicial dga $X$, the associated chain complex of the underlying
simplicial chain complex is a bidga with multiplication defined by
\[CX \tensor CX \xrightarrow{\nabla} C(X \tensor X) \to C(X).\] Similarly,
$NX$ is a bidga. For a bidga $B$, the Alexander-Whitney map induces the
structure of a simplicial dga on $\Gamma(B)$
\cite[(2.9)]{Schwede_Shipley_monoidal}. As explained in \cite[Remark
2.14]{Schwede_Shipley_monoidal}, the functors $\Gamma$ and $N$ do
{\bfseries not} induce an equivalence between $\sdga_k$ and $\bidga_k$, due
to the failure of the unit of the adjunction to be monoidal. However, the
adjunction counit isomorphism $N\Gamma \iso \id_{\biChxes_k}$ is monoidal on
the level of complexes \cite[Lemma 2.11]{Schwede_Shipley_monoidal}, which
is sufficient for our purposes.

\begin{lemma} \label{lem:dold_kan_prop}
\begin{enumerate}[(i)]
\item The functor $C$ maps $\cP$-equivalences of simplicial dgas to
  $E_2$-equivalences of bidgas.
\item For a $k$-bidga $B$, there is a canonical $E_2$-equivalence $C(\Gamma(B))
  \to B$ of bidgas which is natural in $B$. 
\end{enumerate}
\end{lemma}
\begin{proof}
  It is enough to check (i) additively. Since $C$ commutes with taking
  vertical homology, the assertion follows from the definition of
  $\cP$-equivalences and $E_2$-equivalences.

  The map in (ii) is the one inducing the counit $N(\Gamma(B)) \to B$.  It
  is a map of bidgas since the counit is monoidal and the shuffle map is
  compatible with the degeneracies. Both $C$ and $\Gamma$ (but not $N$)
  commute with taking vertical homology, so the Dold-Kan correspondence for
  simplicial $k$-modules and $\Chxes^+_k$ applied in each vertical degree
  shows that $C(\Gamma(B)) \to B$ is an $E_2$-equivalence.
\end{proof}

\begin{proof}[Proof of Theorem \ref{thm:exis_uni_res}]
  We consider $\Gamma(A)$, the constant simplicial dga $A$, and its
  cofibrant replacement $\Gamma(A)^{\textrm{cof}}$.  Lemma
  \ref{lem:eno_Pproj} and Proposition \ref{prop:mstr_sdga} show that
  $\Gamma(A)^{\textrm{cof}}$ has $k$-projective vertical homology.  The
  unit is the composite $k \to C(\Gamma(k)) \to
  C(\Gamma(A)^{\textrm{cof}})$, which splits in vertical homology by the
  last part of Proposition \ref{prop:mstr_sdga}.  Lemma
  \ref{lem:dold_kan_prop} then shows that the composite
  \[ C(\Gamma(A)^{\textrm{cof}}) \to C(\Gamma(A)) \to A\] is a
  $k$-projective $E_1$-resolution. The second statement is the $A=A'$ case
  of the next lemma.
\end{proof}

\begin{lemma} \label{lem:reso_welldef} Given a map $A \to A'$ of $k$-dgas
  and $k$-projective $E_1$-resolutions $B_1 \to A_1$ and $B'_1 \to A'$,
  there is a commutative diagram
  \[ \xymatrix@-1pc{
   B_2 \ar[r] \ar[d] & B_3 \ar[r] \ar[d] & B'_3 \ar[d] & B'_2 \ar[d] \ar[l]  \\
   B_1 \ar[r] & A \ar[r] & A' & B'_1 \ar[l]}
   \]
   of bidgas in which all maps $B_i \to A$ and $B'_i \to A'$
   are $k$-projective $E_1$-resolutions. 
\end{lemma}

\begin{proof}
  We apply the cofibrant replacement functor in the $\cP$-model structure
  on $\sdga_k$ to $ \Gamma(B_1) \to \Gamma(A) \to \Gamma(A') \ot
  \Gamma(B'_1)$ and obtain
  \[ \xymatrix@-1pc{ C(\Gamma(B_1)^{\textrm{cof}}) \ar[r] \ar[d] &
    C(\Gamma(A)^{\textrm{cof}}) \ar[r] \ar[d] &
    C(\Gamma(A')^{\textrm{cof}})
    \ar[d] & C(\Gamma(B'_1)^{\textrm{cof}}) \ar[d] \ar[l] \\
    C(\Gamma(B_1)) \ar[r] \ar[d] & C(\Gamma(A)) \ar[r] \ar[d] &
    C(\Gamma(A'))
    \ar[d] & C(\Gamma(B'_1)) \ar[d] \ar[l] \\
    B_1 \ar[r]& A \ar[r] & A' & B'_1. \ar[l] } \] By Lemma
  \ref{lem:dold_kan_prop}, the upper and the lower row assemble to the
  desired diagram.
\end{proof}

\section{\texorpdfstring{Minimal \dainfalgs}{Minimal derived A-infinity algebras}} \label{sec:minimal} We start by giving an
alternative description of \dainfalgs and their maps.  Disregarding unit
conditions, an $A_{\infty}$-algebra structure on a $\mZ$-graded $k$-module
$M$ can be encoded by giving a differential on the reduced tensor algebra
of the suspension of $A$ \cite{Keller_introduction,Stasheff_associativity}.
Next we explain how the structure of a twisted chain complex on the reduced
tensor algebra (in bigraded $k$-modules) describes a \dainf{}-algebra.

Let $S$ be the shift $[0,1]$ of bigraded $k$-modules. There is a canonical
isomorphism $\sigma \colon S \to \id[0,1]$ of degree $(0,1)$. It induces an
isomorphism
\[
\Psi_j \colon \Hom_k (E^{\tensor j}, F) \to \Hom((SE)^{\tensor j}, SF)
\text{ with } \sigma_F (\Psi_j(f)) = (-1)^{\sclprod{\Psi_j(f)}{\sigma}} f \sigma^{\tensor j}_{E}.
\]
For a $k$-linear map $m_{ij}^E \colon E^{\tensor j} \to E[i, 2-(i+j)]$, we
define $\mtil_{ij}^{E,1} \colon SE^{\tensor j} \to SE[i,1-i]$ to be $\Psi_j
(m_{ij}^E)$.

Let $\Tbar \colon \bimodu{k} \to \bimodu{k}$ be the reduced tensor algebra
functor of bigraded $k$-modules, i.e., $\Tbar E = \bigoplus_{j \geq 1}
E^{\tensor j}$. Let $m_{ij}^E \colon E^{\tensor j} \to E[i, 2-(i+j)]$ be a
family of $k$-linear maps with $i \geq 0$ and $j \geq 1$.  For fixed $i$,
the $\mtil^{E,1}_{ij}$ assemble to a map $\mtil^{E,1}_{i} \colon \Tbar SE
\to SE[i,1-i]$. More generally, we define
\[
\mtil^{E,q}_{ij} := \sum_{\substack{r+1+t=q\\r+s+t=j}} \eins^{\tensor r}
\tensor \mtil^{E,1}_{is} \tensor \eins^{\tensor t} \colon SE^{\tensor j} \to
SE^{\tensor q} [i,1-i]
\]
for $q \leq j$ and write $\mtil^{E}_i \colon \Tbar SE \to \Tbar SE[i,1-i]$
for the map whose component mapping $SE^{\tensor j}$ to $SE^{\tensor
  q}[i,1-i]$ is $\mtil^{E,q}_{ij}$. The process of building the $\mtil^E_i$
from the $\mtil^{E,1}_i$ may also be described using the universal property
of $\Tbar SE$ as a (cocomplete) tensor coalgebra.

\begin{lemma} \label{lem:vari_obj_rel} Let $E$ be a bigraded $k$-module and
  let $m_{ij}^E \colon E^{\tensor j} \to E[i, 2-(i+j)]$ be a family of
  $k$-linear maps with $i \geq 0$ and $j\geq 1$ satisfying the unit
  condition \eqref{equ:obj_unitcondition}.  The following are equivalent:
\begin{enumerate}[(i)]
\item The $m_{ij}^E$ specify a \dainfalg structure on $E$. 
\item $\sum_{i+p=u} (-1)^{i} \mtil^{E}_{i} \mtil^{E}_{p} = 0$ for all $u
  \geq 0$.
\item $\sum_{i+p=u} (-1)^{i} \mtil^{E,1}_{i} \mtil^{E}_{p} = 0$ for all $u
  \geq 0$. 
\item $\Tbar SE$ is a twisted chain complex with $d_i = \mtil^E_i$. 
\end{enumerate}
\end{lemma}
\begin{proof}
  The signs in \ref{equ:obj_rel} arise from interchanging $\sigma$ with the
  $\mtil^{E,1}_{pq}$:
\[\begin{split}
  (-1)^i \mtil^1_{ij}(\eins^{\tensor r} \tensor \mtil^1_{pq} \tensor
  \eins^{\tensor
    t})
  =& (-1)^{p+t(1+p)} \sigma^{-1} m_{ij} (\sigma^{\tensor r} \tensor m_{pq}
  \sigma^{\tensor q} \tensor \sigma^{\tensor t})\\ 
  =& (-1)^{rq+t+pj}
  \sigma^{-1} m_{ij} (\eins^{\tensor r} \tensor m_{pq} \tensor \eins^{\tensor t})
  \sigma^{\tensor r+q+t}
\end{split}\]  
\end{proof}
In terms of the $\mtil^{E,1}_{ij}$, the unit condition
\eqref{equ:obj_unitcondition} may be rephrased as
\begin{equation}\label{equ:obj_mod_unitcondition}
\begin{split}
&\mtil_{01}^{E,1}(\sigma^{-1}_E \eta_E) = 0,\;  
- m_{02}^{E,1}(\sigma^{-1}_E \eta_E \tensor \eins) = \eins = 
\mtil_{02}^{E,1} (\eins \tensor \sigma^{-1}_E\eta_E), \; \text{ and} \\  
&\mtil_{ij}^{E,1}(\eins^{\tensor r-1} \tensor \sigma^{-1}_E \eta_E
\tensor \eins^{\tensor j-r}) = 0
\text{ if } i+j \geq 3 \text{ with } 1 \leq r \leq j
\end{split}
\end{equation}

A similar description applies to maps.  Given bigraded $k$-modules $E$ and
$F$ and $k$-linear maps $f_{pq} \colon E^{\tensor q} \to F[p, 1-(p+q)]$
with $p \geq 0, q \geq 1$, we define the map $\ftil^1_{pq}$ to be
$\Psi_q(f) \colon SE^{\tensor q} \to SF[p,-p]$. The $\ftil^1_{pq}$ assemble
to $\ftil^1_p \colon \Tbar SE \to SF[p,-p]$. We define
\[  \ftil^{j}_{pq}  := \sum_{\substack{p_1 + \dots + p_j = p \\ q_1 + \dots + q_j =q }}
\ftil_{p_1 q_1}^1 \tensor \dots \tensor \ftil_{p_j q_j}^1 \colon SE^{\tensor q} \to SE^{\tensor j}[p,-p]\]
and write $\ftil_p \colon \Tbar SE \to \Tbar
SF[p,-p], p \geq 0,$ for the morphism mapping $SE^{\tensor q}$ to
$SE^{\tensor j}[p,-p]$ by $\ftil^j_{pq}$.

\begin{lemma} \label{lem:vari_map_rel} Let $E$ and $F$ be \dainfalgs and
  let $f_{pq} \colon E^{\tensor q} \to F[p, 1-(p+q)]$ a family of
  $k$-linear maps with $p \geq 0$ and $q \geq 1$ satisfying the unit
  condition \eqref{equ:mor_unitcondition}. The following are equivalent:
\begin{enumerate}[(i)]
\item The $f_{pq}$ form a \dainfalg map from $E$ to $F$. 
\item $\sum_{i+p=u} (-1)^{i} \ftil_i \mtil^E_p = \sum_{i+p=u} \mtil^F_i
  \ftil_p$ for all $u \geq 0$.
\item $\sum_{i+p=u} (-1)^{i} \ftil_i^1 \mtil^E_p = \sum_{i+p=u}
  \mtil^{F,1}_i \ftil_p$ for all $u \geq 0$.
\item The $\ftil_i$ form a map of twisted chain complexes.
\end{enumerate}
\end{lemma}

\begin{remark} \label{rem:signs_explanation} One advantage of using the
  $\mtil^{E}_{ij}$ and the $\ftil^1_{ij}$ is that their bidegrees
  $|\mtil^{E}_{ij}| = (i,1-i)$ and $|\ftil^1_{ij}| = (i,-i)$ do not depend
  on $j$. As one can already see by comparing the last lemma with
  Definition \ref{def:dainf_maps}, this reduces the complexity of signs
  considerably.  Because of this, we will frequently switch to the
  $\mtil_{ij}$ for explicit calculations.

  However, there is still a choice behind our sign convention. For example,
  one may omit the sign in the definition of twisted chain complexes. This
  would change the sign in the definition of \dainfalgs and, in turn, the
  sign in the Leibniz rule for bidgas. This would be in conflict with the
  signs resulting from the Dold-Kan correspondence relating $\biChxes_k$
  and $\sdga_k$. In other words, our sign convention is chosen to ensure
  that the chain complex associated to a simplicial dga is a
  \dainf{}-algebra.
\end{remark}

At this point we can also make up the promised definition of the
composition of \dainfalg maps.

\begin{definition} \label{def:map_comp} Given two maps $g \colon D \to E$
  and $g \colon E \to F$, the component $(\fgtil)_{u}$ of $fg$ is
  $\sum_{i+p=u} \ftil_i \gtil_p$.
\end{definition}

An $A_{\infty}$-algebra is minimal if $m_1^A = 0$. This motivates
\begin{definition} \label{def:minimality}
A \dainfalg $E$ is {\em minimal} if $m_{01}^E = 0$. 
\end{definition}
The device to produce minimal \dainfalgs is the following proposition. It
promotes both the statement and the proof of \cite[Theorem
1]{Kadeishvili_homology} about $A_{\infty}$-algebras to the
context of \dainf{}-algebras.

\begin{proposition} \label{prop:exis_str} Let $B$ be a bidga, let $E$ be
  its homology with respect to $m_{01}^B$, and let $d_1$ and $\mu$ be the
  differential and the multiplication induced on $E$. If $E$ is
  $k$-projective in every bidegree and the unit map $\eta_E \colon k \to E$
  induced by $\eta_B \colon k \to B$ splits as a $k$-module map, then there
  exists
\begin{enumerate}[(i)] 
\item a minimal \dainf{}-structure on $E$ with $m^{E}_{11} = d_1$ and
  $m^{E}_{02} = \mu$ and
\item an $E_2$-equivalence of \dainfalgs $f \colon E \to B$.
\end{enumerate}
\end{proposition}
To construct the structure inductively, we define
\[\label{equ:zln} \eqtagwithsubscr{ln}
\ztil_{ln} = \mtil^{B,1}_{02} \ftil_{ln}^2 + \mtil^{B,1}_{11}
\ftil^1_{l-1,n} - \sum_{\substack{i+p=l\\ i+j \geq 2; \, j<n}}
(-1)^i\ftil^1_{ij} \mtil^{E,j}_{pn}.
\]
The identity \ref{equ:mor_rel} can be rewritten as $\mtil_{01}^{B,1}
\ftil_{uv}^1 = \ftil_{01}^1 \mtil_{uv}^{E,1} - \ztil_{uv} $ by Lemma
\ref{lem:vari_map_rel}.

\begin{proof}[Proof of Proposition \ref{prop:exis_str}]
  Since $E$ is $k$-projective, we can choose a $k$-linear cycle selection
  homomorphism $f_{01} \colon E \to B$ in the diagram
  \[\xymatrix{
    E \ar[d]_{=} \ar@{-->}[dr] \ar@{-->}[drr]^{f_{01}}\\
    H_*(B, m_{01}^B) & \kernel m_{01}^B \ar@{->>}[l] \ar@{->}[r] & B}
  \]
  We may adjust $f_{01}$ to ensure $f_{01}\eta_E = \eta_B$. Setting
  $m_{01}^{E} = 0$, we observe that \ref{equ:mor_rel} holds for $u+v = 1$.

  For the definition of $\mtil^{E,1}_{ln}$ and $\ftil^{1}_{ln}$ we assume
  that we have constructed the $\mtil^{E,1}_{ij}$ and $\ftil^{1}_{ij}$ for
  $i+j \leq l+n$ such that \ref{equ:obj_rel} holds for $u+v \leq l+n$,
  \ref{equ:mor_rel} holds for $u+v < l+n$, and the two unit conditions
  \eqref{equ:obj_unitcondition} and \eqref{equ:mor_unitcondition} are
  satisfied. Then the $\ztil_{ln}$ of \ref{equ:zln} is defined.

  By Lemma \ref{lem:zln_cocycle}, $\ztil_{ln}$ has values in $\kernel
  \mtil_{01}^B$. We define $\mtil_{ln}^{E,1}$ to be the composition of
  $\ztil_{ln}$ with the quotient map $\kernel \mtil_{01}^{B,1} \to SE$.  By
  Lemma \ref{lem:zln_works}, the $\mtil_{ln}^{E,1}$ together with the
  $\mtil^{E,1}_{ij}$ for $i+j < l+n$ satisfy \ref{equ:obj_rel} for $u+v
  \leq l+n+1$. Checking \eqref{equ:obj_unitcondition} for $m_{ln}^{E}$ is
  easy.

  We need to define the $\ftil_{ln}^1$. In each bidegree, $SE^{\tensor n}$
  is $k$-projective as a tensor product of $k$-projective modules. Since
  $\ftil_{01}^1 \mtil_{ln}^{E}$ and $\ztil_{ln}$ both have values in
  $\kernel \mtil_{01}^B$ and coincide when composed with the quotient map
  to $SE$, their difference lies in the image of $\mtil_{01}^B$, and we can
  find lifts $\ftil_{ln}^1$ in
  \[\xymatrix{ (S E)^{\tensor n} \ar[d]_{\ftil_{01}^1 m_{ln}^{E,1}-
      \ztil_{ln}}
    \ar@{-->}[rd]^{\ftil^1_{ln}} \\
    (\image \mtil_{01}^B)[u, 2-(u+v)] & B[u, 1-(u+v)]
    \ar[l]_(.4){\mtil_{01}^{B,1}} }\]
  The $f_{ln}$ satisfy \ref{equ:mor_rel} for $u+v=l+n$ by construction. By
  assumption, the inclusions of the submodules on which the unit condition
  \eqref{equ:mor_unitcondition} requires $f_{ln}$ to be zero splits. Hence
  we may assume $f_{ln}$ to be zero there.
\end{proof}

The input for the last proposition will be a $k$-projective
$E_1$-resolution arising form Theorem \ref{thm:exis_uni_res}. Since these
are only well defined up to $E_2$-equivalence, we need to check that such
an $E_2$-equivalence of bidgas induces a map of the associated minimal
\dainf{}-algebras. This issue does not come up when constructing minimal
models in the world of $A_{\infty}$-algebras, since the underlying graded
algebra of a minimal \ainf{}-model is unique up to isomorphism \cite[\S
8.7]{Keller_introduction}.

To formulate the next proposition, we need to introduce the notion of a
homotopy of \dainfalg maps. We only define it in a special case.

\begin{definition}
  Let $f$ and $g$ be \dainfalg maps from a minimal \dainfalg $E$ to a bidga
  $C$. A homotopy from $g$ to $f$ is a family of maps $\htil_i \colon \Tbar
  SE \to SC[i,-1-i]$ for $i \geq 0$ satisfying
  \[ \label{equ:hty} \eqtagwithsubscr{uv} \begin{split} \gtil_{uv}
    -\ftil_{uv} =& \sum_{i+p=u;\,j+q=v} \left( (-1)^u \mtil_{02}^{C,1} (
      \gtil_{ij}^1 \tensor \htil^1_{pq})+ (-1)^i \mtil_{02}^{C,1} (
      \htil^1_{ij} \tensor
      \ftil_{pq}^1) \right) \\
    & + (-1)^u \mtil_{01}^{C,1} \htil^1_{uv} +(-1)^u \mtil_{11}^{C,1}
    \htil^1_{u-1,v} + \sum_{i+p=u;\, 1 \leq j \leq n} \htil^1_{ij}
    \mtil^{E,j}_{pv}
  \end{split} \] for $u \geq 0$ and $v \geq 1$ and the unit condition
  \begin{equation} \label{equ:hty_unitcondition}
    \htil^1_{ij}(\eins^{\tensor r-1} \tensor \sigma^{-1}_E \eta_E \tensor
    \eins^{\tensor j-r}) = 0 \text{ with } 1 \leq r \leq j \text{ for all }
    i,j.
  \end{equation} 
\end{definition}

\begin{proposition} \label{prop:funct_str} Let $\alpha \colon B \to C$ be a
  map of bidgas with $k$-projective vertical homology algebras $E$ and $F$.
  Let $E$ and $F$ be equipped with \dainfalg structures and \dainfalg maps
  $f \colon E \to B$ and $g \colon F \to C$ arising from Proposition
  \ref{prop:exis_str}.  Then there exist \dainfalg map $\beta \colon E \to
  F$ with $\beta_{01} = H_{*}^v(g)$ and a homotopy from $g \beta$ to
  $\alpha f$.
\end{proposition}
The strategy for the proof is the same as for the last proposition. We
define
\[ \begin{split} 
   \ytil_{ln} 
  =& \mtil_{11}^{C,1} \htil^1_{l-1,n} 
   + (-1)^l \alphaftil^1_{ln} 
   -(-1)^l \hspace{-1.5ex} \sum_{i+j\geq2;\, i+p=l}\hspace{-1.5ex} 
  \gtil_{ij}^1 \betatil^j_{pn} 
   + (-1)^l \hspace{-1.5ex} \sum_{\substack{i+p=l;\, 1 \leq j \leq n}} \hspace{-1.5ex}
  \htil^1_{ij} \mtil^{E,j}_{pn}
  \\
  &+ \sum_{i+p=l;\,j+q=n} \left( \mtil_{02}^{C,1} ( \gbetatil_{ij}^1
    \tensor \htil^1_{pq}) + (-1)^p \mtil_{02}^{C,1} ( \htil^1_{ij} \tensor
    \alphaftil_{pq}^1)\right) \end{split} \] and observe that the identity
\ref{equ:hty} can be rewritten as
\[ \mtil_{01}^C \htil^1_{uv} = (-1)^u \, \gtil^1_{01} \betatil^1_{uv} - \ytil_{uv}. \]

\begin{proof}[Proof of Proposition \ref{prop:funct_str}]
  We set $\beta_{01} = H_*^v(g_{01})$. Since $E$ is $k$-projective,
  there exists $\htil_{01}$ with $\mtil^{C,1}_{01} \htil_{01} =
  \alphatil^1_{01} \ftil^1_{01} - \gtil_{01}^1 \betatil^1_{01}$. The
  $\beta_{01}$ satisfies \ref{equ:mor_rel} for $u+v \leq 2$ and
  \eqref{equ:mor_unitcondition}, and $\htil_{01}$ satisfies \ref{equ:hty}
  for $u+v=1$ and can be chosen to satisfy \eqref{equ:hty_unitcondition}.
  
  Assume that we have constructed $\betatil_{ij}$ and $h_{ij}$ for $i+j<
  l+n$ such that \ref{equ:mor_rel} holds for $ u+v \leq l+n$ and
  \ref{equ:hty} holds for $u+v < l+n$. Then $y_{ln}$ is defined, and Lemma
  \ref{lem:yln_cycle} provides $\mtil^{C,1}_{01} \ytil_{ln}=0$. So we
  define $\betatil^1_{ln}$ to be represented by $\ytil_{ln}$. Lemma
  \ref{lem:yln_works} shows that \ref{equ:mor_rel} holds for $u+v=l+n+1$,
  and the unit condition \eqref{equ:mor_unitcondition} can be checked
  readily.  Since $E$ is $k$-projective we can find $\htil_{ln}$ such that
  \ref{equ:hty} holds for $u+v=l+n$. The splitting of the unit of $E$
  enables us to change $\htil_{ln}$ such that \eqref{equ:hty_unitcondition}
  holds.
\end{proof}

We can now give the proof of the first main theorem from the introduction.

\begin{proof}[Proof of Theorem \ref{thm:min_models}]
  Given a $k$-dga $A$, we apply Theorem \ref{thm:exis_uni_res} to obtain a
  $k$-projective $E_1$-resolution $B \to A$. Then $B$ satisfies the
  assumptions of Proposition \ref{prop:exis_str}, which provides an
  $E_2$-equivalence $E \to B$. The composition $E \to A$ is the desired
  minimal \dainfalg model.

  The resolution $B \to A$ is well defined up to $E_2$-equivalence of
  $k$-projective $E_1$-resolutions. Given such an $E_2$-equivalence $B \to
  C$, Proposition \ref{prop:funct_str} provides an $E_2$-equivalence $E \to
  F$ between the associated minimal \dainf{}-algebras.
\end{proof}

Lemma \ref{lem:reso_welldef} and Proposition \ref{prop:funct_str} imply

\begin{corollary}
Minimal \dainfalg models of quasi-isomorphic dgas may be related by a
zig-zag of $E_2$-equivalences between minimal \dainf{}-algebras. 
\end{corollary}

\begin{remark}
We do not claim that every termwise $k$-projective resolution $E$ of
$H_*(A)$ may be extended to a minimal \dainfalg model for $A$. We only show
that this happens if $E$ is the vertical homology of a bidga resolving
$A$. 
\end{remark}

\begin{remark}
  The underlying twisted chain complex of a minimal \dainfalg model of $A$
  is a `homological multicomplex resolution' of the underlying chain
  complex of $A$; compare \cite{Saneblidze_derived}.
\end{remark}

\section{Examples and an application} \label{sec:ex_appl}

In this section, we give two examples for minimal \dainfalg structures and
apply our theory to define the derived Hochschild cohomology class
associated with a dga. 

\begin{example} \label{ex:minimalmodel} Let $p$ be an odd prime, $k =
  \mZ/p^2$, and $A$ the dga $(k \xrightarrow{\cdot p} k)$.  In Example
  \ref{ex:E1res}, we gave a $k$-projective $E_1$-resolution $B \to A$ with
  $E := H_*^v(B) = \Lambda_k^*(\ovl{x}, \ovl{z}) \tensor
  \Gamma_k^*(\ovl{y})$.  Following the proof of Proposition
  \ref{prop:exis_str}, we now extend the multiplication and $d_1^E$ to a
  \dainfalg structure on $E$.  A cycle selection map $f_{01}$ can be
  defined by extending
  \[ f_{01}(\ovl{z}) = au, f_{01}(\ovl{x}) = x, \textrm{ and }
  f_{01}(\ovl{y}_i) = y_{i}\] multiplicatively. Then $f_{01} m_{02}^{E} -
  m_{02}^B(f_{01} \tensor f_{01}) = 0$. Consequently, $m_{03}^E=0$, and we
  can set $f_{02}=0$ and $f_{03}=0$. Since $m_{11}^{E} = d^{E}_1$, we
  define $f_{11} \colon E \to B[1,-1]$ by
  \[ f_{11}(y_i) = (v+av) x y_{i-1} \text{ and } f_{11}(\ovl{x} \,
  \ovl{y}_i) = (pu-v)y_i. \] The $m^{E}_{12}$ is represented by \[z_{12}
  = f_{11} m_{02}^E-m_{02}^B(f_{01} \tensor f_{11} + f_{11}\tensor f_{01}).
  \]  
  One calculates $z_{12}=0$, hence $m_{12}^E=0$ and we can set $f_{12}=0$.
  The $m_{21}^E$ is represented by $z_{21}=f_{11} m_{11}^E + m_{11}^B f_{11}$. One
  checks
\[ m_{21}^{E}(\ovl{y}_i) = \ovl{z} \, \ovl{y}_{i-1} \text{ and }
m_{21}^{E}(\ovl{x} \, \ovl{y}_i) = \ovl{z} \, \ovl{x} \, \ovl{y}_{i-1}.
\] Furthermore, we set $f_{21}(\ovl{x}\,\ovl{y}_i)=puvxy_{i-1}$ and $f_{21}(\ovl{y}_i)=0$.
The $m^E_{ij}$ and $f_{ij}$ for $i+j \geq 4$ vanish for degree reasons.

  Example \ref{ex:shukla class} below provides one way to detect that this
  $E$ encodes non-trivial information about $A$, for example the fact that
  $A$ is not quasi-isomorphic to the formal dga $H_*(A)$ with trivial
  differential. Of course, Theorem \ref{thm:antimin_models} provides a
  different reason this.
\end{example}

The example indicates the general procedure for finding a minimal \dainfalg
model of a given dga $A$. Since the resolutions provided by Theorem
\ref{thm:exis_uni_res} are too large to write down in terms of generators
and relations, one has to first guess a $k$-projective $E_1$-resolution,
for example by lifting a $k$-projective resolution of $H_*(A)$ to a bidga
$B$ coming with an $E_2$-equivalence $B \to A$. Then one can use the
constructive procedure of Proposition \ref{prop:exis_str} to build $E$ from
$B$.

\begin{remark} \label{rem:gen_kadeishvili} If $k$ happens to be a field or,
  more generally, $H_*(A)$ is $k$-projective and the unit $k \to H_0(A)$
  splits, the minimal model of $A$ as an $A_{\infty}$-algebra arising from
  \cite[Theorem 1]{Kadeishvili_homology} is one possible choice for the
  minimal \dainfalg model of Theorem \ref{thm:min_models}.
\end{remark}

Let $M$ be a $k$-module and let $P$ be a $k$-projective
resolution of $M$.  The endomorphism dga $\Hom_k(P,P)$ is defined by
$\Hom_k(P,P)_j = \Hom_k(P[j],P)$, where the latter denotes maps of graded
$k$-modules. Its differential is $d^{\Hom}_j(f)= d^P f - (-1)^j f d^P$.
The quasi-isomorphism type of $\Hom_k(P,P)$ does not depend on the choice
of the resolution $P$.  Its homology algebra is the Yoneda $\Ext$-algebra of $M$
with a sign shift in the degree, i.e., $H_* \Hom_k(P,P) = \Ext^{-*}_k(M,M)$.

A minimal \dainf{}-model for $\Hom_k(P,P)$ is a resolution $E$ of
$\Ext_k^*(M,M)$ with structure maps $m_{ij}^E$. By Theorem
\ref{thm:antimin_models}, this structure enables us to recover the
quasi-isomorphism type of $\Hom_k(P,P)$.

\begin{example}
Let $k=\mZ$, $p$ a prime, and $M=\mZ/p$. We will see that the resolution of
$\Ext^*_{\mZ}(\mZ/p,\mZ/p)$ given by $\mZ \xleftarrow{\cdot p} \mZ$ in
every degree of the $\Ext$-grading is part of a minimal \dainf{}-algebra. 

Let $P$ be the two step resolution $\mZ \xleftarrow{\cdot p} \mZ$ of $\mZ/p$.
The dga $A=\Hom_{\mZ}(P,P)$ may be described as follows (compare \cite[\S
3.1]{Dwyer_G_complete}): Its underlying ungraded algebra is $\Mat_2(\mZ)$,
the algebra of $2 \times 2$-matrices with entries in $\mZ$. Viewing
elements of $P$ as column vectors with top entry the degree $0$ part,
$\Mat_2(\mZ)$ operates on $P$. The resulting grading of $A$ is
\[ 
A_1 = \left\{ \bigl( 
\begin{smallmatrix} 0&0 \\ *&0 \end{smallmatrix} 
\bigr) \right\}, 
\quad 
A_0 = \left\{ \bigl( 
\begin{smallmatrix} *&0 \\ 0&* \end{smallmatrix} 
\bigr) \right\}, \text{ and} \quad 
A_{-1} = \left\{ \bigl( 
\begin{smallmatrix} 0&* \\ 0&0 \end{smallmatrix} 
\bigr) \right\}. 
\] 
On the standard basis, the differential has values
\[ 
d\bigl( 
\begin{smallmatrix} 0&0 \\ 1&0 \end{smallmatrix} 
\bigr)  
 = 
\bigl( 
\begin{smallmatrix} p&0 \\ 0&p \end{smallmatrix} 
\bigr),  
\quad 
d\bigl( 
\begin{smallmatrix} 1&0 \\ 0&0 \end{smallmatrix} 
\bigr)  
 = 
-\bigl( 
\begin{smallmatrix} 0&p \\ 0&0 \end{smallmatrix} 
\bigr),  
\quad 
d\bigl( 
\begin{smallmatrix} 0&0 \\ 0&1 \end{smallmatrix} 
\bigr)  
 = 
\bigl( 
\begin{smallmatrix} 0&p \\ 0&0 \end{smallmatrix} 
\bigr), \text{ and}  
\quad 
d\bigl( 
\begin{smallmatrix} 0&1 \\ 0&0 \end{smallmatrix} 
\bigr)  
 = 0.
\]
Indeed, $H_*(A)$ is an exterior algebra generated by 
$\left[ \bigl( 
\begin{smallmatrix} 0&1 \\ 0&0 \end{smallmatrix} 
\bigr)\right]$ in degree $-1$.

We give a $k$-projective $E_1$-resolution $\alpha \colon B \to A$. Let $B =
\Mat_2(\mZ[t] \tensor \Lambda^*_{\mZ}(\lambda))$ with $|t|=(0,0)$ and
$|\lambda| = (1,0)$. The horizontal grading is specified through the
generators, and the vertical grading comes from the matrix algebra as
above.

The vertical differential $d^B_0$ is defined by setting 
\[ 
d_0^B\bigl( 
\begin{smallmatrix} 0&0 \\ 1&0 \end{smallmatrix} 
\bigr)  
 = 
\bigl( 
\begin{smallmatrix} t&0 \\ 0&t \end{smallmatrix} 
\bigr),  
\quad 
d_0^B\bigl( 
\begin{smallmatrix} 1&0 \\ 0&0 \end{smallmatrix} 
\bigr)  
 = 
-\bigl( 
\begin{smallmatrix} 0&t \\ 0&0 \end{smallmatrix} 
\bigr),  
\quad 
d_0^B\bigl( 
\begin{smallmatrix} 0&0 \\ 0&1 \end{smallmatrix} 
\bigr)  
 = 
\bigl( 
\begin{smallmatrix} 0&t \\ 0&0 \end{smallmatrix} 
\bigr), \text{ and}  
\quad 
d_0^B\bigl( 
\begin{smallmatrix} 0&1 \\ 0&0 \end{smallmatrix} 
\bigr)  
 = 0,
\]
and requiring $d_0^B$ to be $t$- and $\lambda$-linear. The horizontal
differential $d_1^B$ is $t$-linear and maps $\lambda$ to $p-t$. The map
$\alpha$ is determined by asking $t \mapsto p$.

One can check that $B \to A$ has the required properties. The vertical
homology $H_*^v(B) = E$ is $\Lambda_{\mZ}^*(a,b)$ with $|a|=(0,-1)$ and
$|b|=(1,0)$. Here
\[
\iota=
\left[\bigl( 
\begin{smallmatrix} 1&0 \\ 0&1 \end{smallmatrix} 
\bigr)\right],  
\quad 
a=
\left[\bigl( 
\begin{smallmatrix} 0&1 \\ 0&0 \end{smallmatrix} 
\bigr)\right],  
\quad 
b=
\left[\bigl( 
\begin{smallmatrix} \lambda &0 \\0 &\lambda \end{smallmatrix} 
\bigr)\right], \text{ and}  
\quad 
ab=
\left[\bigl( 
\begin{smallmatrix} 0 &\lambda\\ 0 &0\end{smallmatrix} 
\bigr)\right],  
\quad 
\] 
and $f_{01}$ can be chosen to pick the displayed generators. The $m_{02}^E$
provides the multiplicative structure encoded in the exterior
algebra. Since $f_{01}$ is multiplicative, we may choose $f_{02}=0$ (and
$f_{0j}=0, m_{0j}^E=0$ for $j \geq 3$). Now $z_{11} = m_{11}^B f_{01}$,
hence $m_{11}^E(b)=p\iota$ and $m_{11}^E(ab) = pa$. Since $m_{01}^B f_{11}
= f_{01} m_{11}^E - z_{11}$, we may choose 
\[
f_{11}(b) = 
\bigl( 
\begin{smallmatrix} 0&0 \\ 1&0 \end{smallmatrix} 
\big) 
\quad \text{and} \quad 
f_{11}(ab) = 
\bigl( 
\begin{smallmatrix} 0&0 \\ 0&1 \end{smallmatrix} 
\big) 
\]
The $m_{12}$ is represented by $z_{12}= f_{11} m_{02}^E -
m_{02}^B(f_{01}\tensor f_{11}) - m_{02}^B(f_{11}\tensor f_{01})$. One
checks that
\[ m_{12}^E(a \tensor b) = \iota,\quad m_{12}^E(a \tensor ab) = a,\quad
m_{12}^E(ab \tensor b) = -b, \text{ and} \quad m_{12}^E(ab \tensor ab) =
-ab,\] and that $m_{12}^E$ is trivial elsewhere. One calculates that,
mostly for degree reasons, all other $m_{ij}^E$ vanish. The asymmetry of
$m_{12}^E$ comes from the choice of the value of $f_{11}(ab)$.  Changing it
to $-\bigl( \begin{smallmatrix} 1&0 \\ 0& 0\end{smallmatrix} \big)$ changes
$m_{12}^E$.
\end{example} 

Our next aim is to define a variant of {\em derived Hochschild cohomology} for
graded $k$-algebras. It will generalize the Hochschild cohomology of graded
$k$-algebras which is for example used in \cite{Benson_K_S_realizability}.

We briefly recall from \cite[\S 4]{Baues_P_shukla} one way to define
the derived Hochschild cohomology $\dHH$ for ungraded
$k$-algebras. (In \cite{Baues_P_shukla}, this theory is called {\em
  Shukla cohomology}.)  Considering an ungraded $k$-algebra $\Lambda$
as a $\mN$-graded dga concentrated in degree $0$, one can find a
quasi-isomorphism $A \to \Lambda$ from a degreewise $k$-projective dga
$A$.  The Hochschild complex of $A$ is the bicomplex $C^{st}(A) :=
\Hom_k(A^{\tensor s}, A[t])$ with $s,t \geq 0$. It has the usual
Hochschild differential $d_{\HH} \colon C^{st} \to C^{s+1,t}$ and a
differential $d_{\Hom} \colon C^{st} \to C^{s,t+1}$ induced from that
of $A$ by considering $C^{s*}$ as the $\Hom$-complex with source
$A^{\tensor s}$ and target $A$. Letting $\HH^*(A)$ be the cohomology
of the total complex of $C^{st}$, we define $\dHH^*(\Lambda) :=
\HH^*(A)$. This is well defined since a quasi-isomorphism of
$k$-projective dgas $A \to A'$ over $\Lambda$ induces a chain of
isomorphisms $\HH^*(A) \xrightarrow{\iso} \HH^*(A,A')
\xleftarrow{\iso} \HH^*(A')$; compare \cite[4.1.1
Lemma]{Baues_P_shukla}.

Let $E$ be a $k$-bidga with $d_0^E=0$, and let $F$ be an $E$-bimodule in
the category $\biChxes_k$ with $d^F_0=0$. On $E$ and $F$ we have
differentials $d_1^E$ and $d_1^F$ and multiplication maps $E \tensor E \to
E$, $E \tensor F \to F$, and $F \tensor E \to F$ satisfying appropriate
relations.

We define an $(\mN,\mN,\mZ)$-trigraded $k$-module $C^{rst}(E,F) = \Hom_k(
E^{\tensor r}, F[s,t])$. It has a Hochschild differential
\[ \begin{split} d^r_{\HH}
  \colon &C^{rst}(E,F) \to C^{r+1,st}(E,F) \\
  &c \mapsto \mu(\eins \tensor c) + (-1)^{r+1} \mu(c \tensor \eins) +
  \sum_{1\leq j \leq r} (-1)^j c(\eins^{\tensor j-1} \tensor \mu \tensor
  \eins^{\tensor r-j}) \end{split} \] 
and a differential
\[ \begin{split} 
d^s_{\Hom} \colon &C^{rst} \to C^{r,s+1,t}\\
&c \mapsto d_1^F c - (-1)^{\langle d_1^E, c\rangle}  \sum_{1\leq j \leq r} c
(\eins^{\tensor j-1} \tensor d_1^E \tensor \eins^{\tensor w-j}) 
\end{split} \] induced by the horizontal differentials of $E$ and $F$. It is
easy to check $d_{\HH} d_{\HH} = 0$, $d_{\Hom} d_{\HH} = d_{\HH} d_{\Hom}$ and
$d_{\Hom} d_{\Hom} = 0$.

We define $\Tot^{qt} C(E,F) = \bigoplus_{r+s=q} C^{rst}(E,F)$ and
\[ d^q_{\Tot} \colon \Tot^{qt}C(E,F) \to \Tot^{q+1}C(E,F), \quad c\mapsto
d_{\Hom}(c) - (-1)^q d_{\HH}(c)
\]
The relations of $d_{\HH}$ and $d_{\Hom}$ imply $d_{\Tot} d_{\Tot} = 0$.
The graded Hochschild cohomology of $E$ with coefficients in $F$ is the
homology of this total complex, i.e., $\HH^{qt}(E,F) =
H^q(\Tot^{*t}C(E,F))$.

Now let $\Lambda$ be a graded $k$-algebra. We view $\Lambda$ as a
$k$-dga with trivial differential. Applying vertical homology to the
resolution provided by Theorem \ref{thm:exis_uni_res}, we obtain a
termwise $k$-projective resolution $E \to \Lambda$. The resolution is
well defined up to $E_2$-equivalence between such $E$. We define the
{\em derived Hochschild cohomology} of $\Lambda$ by
\[\dHH^{qt}(\Lambda) = \HH^{qt}(E,E).\] 
An $E_2$-equivalence $\alpha \colon E \to F$ of termwise $k$-projective
resolutions induces a chain of isomorphisms $\HH^{qt}(E,E)
\xrightarrow{{\alpha}_*} \HH^{qt}(E,F) \xleftarrow{\alpha^*}
\HH^{qt}(F,F)$. This ensures that $\dHH^{qt}(\Lambda)$ is well defined.

\begin{proposition} \label{prop:sh_class} Let $A$ be a $k$-dga, and let $E
  \to A$ be a minimal \dainfalg model of $A$. Then the sum $m_{03}^E +
  m_{12}^E + m_{21}^E$ is a cocycle in the complex computing
  $\HH^{3,-1}(E,E) = \dHH^{3,-1}(H_*(A))$. Its cohomology class $\gamma_A \in
  \dHH^{3,-1}(H_* A)$ does not depend on the choice of $E$.
\end{proposition}
\begin{proof}
  The four formulas \ref{equ:obj_rel} with $u+v=4$ for the $m_{ij}^E$ can
  be expressed as $d_{\HH}(m_{03}^E) = 0$, $d_{\HH}(m_{12}^E) +
  d_{\Hom}(m_{03}^E) = 0$, $d_{\HH}(m_{21}^E) + d_{\Hom}(m_{12}^E) = 0$,
  and $d_{\Hom}(m_{21}^E)=0$.  Hence $d_{\Tot}(m_{03}^E + m_{12}^E +
  m_{21}^E)=0$.

  Let $f \colon E \to F$ be an $E_2$-equivalence of minimal
  \dainf{}-algebras. We check that the elements $(f_{01})_*(m_{03}^E +
  m_{12}^E + m_{21}^E)$ and $(f_{01})^* (m_{03}^F + m_{12}^F + m_{21}^F)$
  of $\Tot^{3,-1}C(E,F)$ represent the same cohomology class. Since $f$ is
  map of \dainfalgs, it satisfies the relations \ref{equ:mor_rel} for
  $u+v=3$, and these can be rewritten as
  \[ (f_{01})_* (m_{03}^E + m_{12}^E + m_{21}^E) - (f_{01})^* (m_{03}^F +
  m_{12}^F + m_{21}^F) = d_{\Tot}(f_{02}+f_{11}). \]
\end{proof}

If a dga $A$ is formal, i.e., quasi-isomorphic to $H_*(A)$ with trivial
differential, then $\gamma_A=0$ since a termwise $k$-projective resolution
of $H_*(A)$ provides a minimal \dainfalg model for $A$ with $m_{ij} =$ for
$i+j \geq 3$. Hence a non-vanishing $\gamma_A$ shows that $A$ is not
formal, and that a given minimal \dainfalg model of $A$ is not
equivalent to one with trivial higher $m_{ij}$.

\begin{remark}
  For a dga over a field, there is a characteristic Hochschild cohomology
  class $\gamma_A \in \HH^{3,-1}(H_*(A))$ which is studied in
  \cite{Benson_K_S_realizability}. It is represented by the $m_3$ of a
  minimal $A_{\infty}$-algebra structure on $H_*(A)$ and determines by
  evaluation all triple matric Massey products in $H_*(A)$. The
  characteristic $\dHH$-class of a dga over a commutative ring $k$ given by
  the previous proposition is the natural generalization. If $k$ happens to
  be a field, the $\gamma_A$ introduced here coincides with the one from
  \cite{Benson_K_S_realizability}, as one can easily see from Remark
  \ref{rem:gen_kadeishvili} and \cite[Remark
  5.8]{Benson_K_S_realizability}.

  If $B$ is a bidga and $B \to A$ an $E_2$-equivalence, $B_{0*}$ is a dga,
  and $B \to A$ restricts to a map $B_{0*} \to A$ of dgas which induces an
  surjection in homology. If $E$ is the minimal \dainfalg structure on
  $H_*^v(B)$, the restriction of $m_{03}^E$ to $H_*(B_{0*})$ is part of an
  $A_{\infty}$-structure on $H_*(B_{0*})$. By evaluation, it determines
  triple Massey products of $B_{0*}$. Moreover, if $\lambda_1, \lambda_2,
  \lambda_3 \in H_*(A)$ satisfy $\lambda_1 \lambda_2 = 0 = \lambda_2
  \lambda_3$, we may lift them to elements of $H_*(B_{0*})$, evaluate
  $m_{03}^E$ on the lifts, and observe that the image of the evaluation
  under the map to $H_*(A)$ is an element of the Massey product $\langle
  \lambda_1, \lambda_2, \lambda_3 \rangle$. Hence the $\dHH$-class
  determines all triple (matric) Massey products in $H_*(A)$. The
  difference to the $\HH$-class of \cite{Benson_K_S_realizability} is that
  the $\dHH$-class does not determine the Massey products in a $k$-linear
  fashion, since lifts of elements along a surjection into a not necessarily
  projective $k$-module are involved.

  We expect the $\dHH$-class $\gamma_A$ of a dga to share more properties
  with the $\HH$-class of \cite{Benson_K_S_realizability}, like the
  connection to realizability obstructions.  We are not going to elaborate
  those here.  The $\dHH$-class should also be compared to the universal
  Toda bracket of a ring spectrum studied in \cite{Sagave_universal}. The
  latter class is an even further generalization of the (derived)
  Hochschild class and arises when one allows the sphere spectrum of stable
  homotopy theory as ground ring. This class is an element in a Mac Lane
  cohomology group and determines triple Toda brackets.
\end{remark}

\begin{remark}
  Derived Hochschild cohomology classes associated to dgas also appear in
  \cite{Baues_P_shukla} and \cite{Dugger_S_topological}: a dga $A$ with
  only $A_0$ and $A_1$ possibly non-zero gives rise to a crossed extension
  of $H_0(A)$ by $H_1(A)$ and hence to a cohomology class in
  $\dHH_k^3(H_0(A),H_1(A))$ \cite[4.4.1. Theorem]{Baues_P_shukla}. More
  generally, for a dga $A$ with non-negative homology, its first
  $k$-invariant lies in $\dHH_k^3(H_0(A),H_1(A))$ \cite[\S
  4]{Dugger_S_topological}.

  If $E \to A$ is a minimal \dainfalg model of $A$, $E_{*0}$ is a
  $k$-projective resolution of $H_0(A)$, and there is a restriction map
  $\dHH_k^{3,-1}(H_*(A)) \to \dHH_k^3(H_0(A),H_1(A))$.  One can check that
  the image of $\gamma_A$ under this map recovers the cohomology classes
  described above.
\end{remark}

\begin{example} \label{ex:shukla class} We continue Example
  \ref{ex:minimalmodel} by applying Proposition \ref{prop:sh_class} to
  the minimal \dainfalg $E$ described there. In this case, the sum
  representing $\gamma_{A}$ consists only of $m_{21}^{E}$. The
  $m_{21}^{E}$ cannot be a boundary in the complex computing the
  derived Hochschild cohomology: $p\cdot m_{11}^{E}=0$, while $p\cdot
  m_{21}^{E}$ is non-zero. Hence $\gamma_A \in
  \dHH^{3,-1}_{\mZ/p^2}(H_* A)$ is non-trivial.

  This illustrates that the non-triviality of $E$ is a consequence of the
  information encoded in the higher degrees of the horizontal grading.
\end{example}

\section{Anti-minimal models}\label{sec:anti}
The aim of this section is to prove Theorem \ref{thm:antimin_models}, which
provides an `anti-minimal' model of a \dainf{}-algebra. To this end, we
introduce modules over \dainfalgs and show that the module category is
enriched in twisted chain complexes. For a \dainfalg $E$, we thus have an
endomorphism tga of the $E$-module $E$. Its total complex is a dga. If $E$
is the minimal \dainfalg model of a dga $A$, this dga associated with $E$
recovers the quasi-isomorphism type of $A$. These constructions are
motivated by the $\Hom$-complex of maps between modules over
$A_{\infty}$-algebras described in Keller's survey \cite[\S
6.3]{Keller_introduction} and studied in more detail in Lef\`evre's thesis
\cite[\S 4]{Lefevere_sur}.

Let $E$ be a \dainf{}-algebra. We write $TSE$ for the unreduced tensor
algebra on $SE$, i.e., $TSE = \bigoplus_{j \geq 0} SE^{\tensor j}$.  Let
$M$ be a bigraded $k$-module. Let $\mtil^{M,1}_{ij} \colon SM \tensor
SE^{\tensor j-1} \to SM[i,1-i]$ be a family of maps with $i\geq 0$ and $j
\geq 1$. These maps assemble to a map $\mtil^{M,1}_i \colon SM \tensor TSE
\to SM[i,1-i]$. More generally, we define
\[ 
\mtil^{M,q}_{ij} := \mtil^{M,1}_{i,j-q+1} \tensor \eins^{\tensor q-1} +
\eins \tensor \mtil^{E,q-1}_{i,j-1}
\]
where the second summand is understood to be zero if $q=1$. We obtain a map
$\mtil^{M}_i \colon SM \tensor TSE \to SM \tensor TSE[i,1-i]$ whose
component mapping $SM \tensor SE^{\tensor j-1}$ to $SM \tensor SE^{\tensor
  q-1}[i,1-i]$ is $\mtil^{M,q}_{ij}$.

\begin{lemma}\label{lem:vari_mod_rel}
  Given maps $\mtil^{M,1}_{ij}$ as above, the following are equivalent:
\begin{enumerate}[(i)]
\item $\sum_{i+p=u} (-1)^{i} \mtil^{M}_{i} \mtil^{M}_{p} = 0$ for all $u
  \geq 0$.
\item $\sum_{i+p=u} (-1)^{i} \mtil^{M,1}_{i} \mtil^{M}_{p} = 0$ for all $u
  \geq 0$.
\item The $\mtil^{M}_i$ turn $SM \tensor TSE$ into a twisted chain complex.
\end{enumerate}
\end{lemma}

\begin{definition}\label{def:modules}
  A module over a \dainfalg $E$ is a bigraded $k$-module $M$ together with
  maps $\mtil^{M,1}_{ij} \colon SM \tensor SE^{\tensor j-1} \to SM[i,1-i]$
  for $i \geq 0$ and $j\geq 1$ satisfying the equivalent conditions of the
  last lemma and the unit condition
  \begin{equation} \label{equ:mod_unitcondition} \begin{split}
      &\mtil^{M,1}_{02}(\eins \tensor \sigma_E^{-1} \eta_E) = \eins
      \quad \text{ and } \\
      &\mtil^{M,j}_{ij}(\eins \tensor \eins^{\tensor r-1} \tensor
      \sigma^{-1}_E \eta_E \tensor \eins^{\tensor j-1-r}) = 0 \text{ if } i+j
      \geq 3 \text{ with } 1 \leq r \leq j-1.
    \end{split}
  \end{equation} 
\end{definition}

\begin{example}\label{ex:free_modul}
  The free module of rank $1$ is an $E$-module. For this, the isomorphism
  $\Tbar SE \iso SE \tensor TSE$ is used to interpret the structure maps of
  $E$ as an \dainfalg as those of $E$ as an $E$-module.
\end{example}

Let $f \colon E \to F$ be a \dainfalg map and let $M$ be an $F$-module.
Writing $\mtil^{M,F}_i$ for the $F$-module structure maps of $M$, we define
\[ \mtil^{M,E,1}_{uv} =  \hspace{-1.5ex} \sum_{\substack{i+p=u;\, 1 \leq j \leq v}} \hspace{-1.5ex}
\mtil^{M,F,1}_{ij}(\eins \tensor \ftil^{j-1}_{p,v-1}) \colon SM \tensor
(SE)^{\tensor v-1} \to SM[u,1-u] \]

\begin{lemma}\label{lem:induced_modstr}
  The $\mtil^{M,E}_u$ define an $E$-module structure on $M$.
\end{lemma}

\begin{example}\label{ex:induced_modstr} 
  If $f \colon E \to F$ is a map of \dainf{}-algebras, $F$ is an $E$-module.
\end{example}

Let $E$ be a \dainfalg and let $M$ and $N$ be $E$-modules.  Given a family
of $k$-linear maps $\ftil^1_{ij} \colon SM \tensor SE^{\tensor j-1} \to
SN[i,-i]$ with $i\geq 0$ and $j\geq 1$, they induce a map $\ftil^1_i \colon
SM \tensor TSE \to SM[i,-i]$. Setting $\ftil^q_{ij} = \ftil^1_{i,j-q+1}
\tensor \eins^{\tensor q-1}$ for $1 \leq q \leq j$, we obtain maps $\ftil_i
\colon SM \tensor TSE \to SN \tensor TSE[i,-i]$ whose component mapping $SM
\tensor SE^{\tensor j-1}$ to $SN \tensor SE^{\tensor q-1}[i,-i]$ is
$\ftil^q_{ij}$.

\begin{lemma}\label{lem:vari_modmap_rel}
  Given maps $\ftil^1_{ij}$ as above, the following are equivalent:
\begin{enumerate}[(i)]
\item $\sum_{i+p=u} (-1)^{i} \ftil_i \mtil^M_p = \sum_{i+p=u} \mtil^N_i
  \ftil_p$ for all $u \geq 0$.
\item $\sum_{i+p=u} (-1)^{i} \ftil_i^1 \mtil^M_p = \sum_{i+p=u}
  \mtil^{N,1}_i \ftil_p$ for all $u \geq 0$.
\item The $\ftil_i$ form a map of twisted chain complexes.
\end{enumerate}
\end{lemma}

\begin{definition}\label{def:modulemaps}
  Let $E$ be a \dainf{}-algebra, and let $M$ and $N$ be $E$-modules. An
  $E$-module map from $E$ to $F$ is a family of maps $\ftil^1_{ij} \colon
  SM \tensor SE^{\tensor j-1} \to SN[i,-i]$ for $i \geq 0$ and $j \geq 1$
  satisfying the equivalent conditions of the last lemma and
\[
\ftil_{ij}^1 (\eins \tensor \eins^{\tensor r-1} \tensor \sigma_E^{-1}
\eta_E \tensor \eins^{\tensor j-1-r}) = 0 \text{ if } j \geq 2 \text{ with
} 1 \leq r \leq j-1
\]
Composition of module maps is given by the composition of the $\ftil_i$ as
maps of twisted chain complexes.
\end{definition}

A map $f\colon M \to N$ of $E$-modules induces a map of twisted chain
complexes between the underlying twisted chain complexes of $M$ and $N$.
The map $f$ is an $E_2$-equivalence if this underlying map is.

\begin{example}\label{ex:module_maps}
  If $f\colon E \to F$ is a map of \dainf{}-algebras, both $E$ and $F$ are
  $E$-modules, and one can check that $f$ is an $E$-module map. The
  underlying twisted chain complexes of $E$ and $F$ as \dainfalgs and
  modules coincide. Hence $f$ is an $E_2$-equivalence of modules if and
  only if it is an $E_2$-equivalence of \dainf{}-algebras.
\end{example}

Let $E$ be a \dainfalg and let $M$ and $N$ be $E$-modules. Let $X$ be the
bigraded $k$-module $\Hom_k(SM \tensor TSE, SN)$. For $s \in \mN, t \in
\mZ$, an element of $X_{st}$ is a $k$-module map $\gtil^1 \colon SM \tensor
TSE[s,t] \to SN$.  We shift $X$ by shifting the second entry $SN$. As in
the definition of module maps, the $\gtil^1$ has components $\gtil^1_j$
which provide $\gtil^q_j = \gtil^1_{j-q+1}\tensor \eins^{\tensor q-1}$ and
$\gtil \colon SM \tensor TSE[s,t] \to SN \tensor TSE$, and giving $\gtil^1$
is equivalent to giving the $\gtil$.

For $i \geq 0$, we define $d_i \colon X \to X[i,1-i]$ by $d_i(\gtil^1) =
\mtil^{N,1}_{i} \gtil - (-1)^{\sclprod{\gtil}{\mtil_i}} \gtil^1
\mtil^{M}_i$. Since
\[\sum_{i+p=u} (-1)^p d_i d_p (\gtil^1) = 0 \] for $u \geq 0$,
we may state
\begin{definition}\label{def:twcx_enr}
Let $E$ be a \dainfalg and let $M$ and $N$ be $E$-modules. We define 
$\enrhom_E(M,N)$ to be the twisted chain complex with underlying bigraded
$k$-module $\Hom_k(SM \tensor TSE, SN)$ and differentials $d_i$ as defined
above.
\end{definition} 

Let $L, M$ and $N$ are $E$-modules. For $\ftil \in \enrhom_E(M,N)_{st}$ and
$\gtil \in \enrhom_E(L,M)_{pq}$, its composite $\ftil \gtil \in
\enrhom_E(L,M)_{p+s,q+t}$ is defined. The unit $k \to \enrhom_E(M,M)$ sends
$1_k$ to the $\gtil \in \enrhom_E(M,M)$ with $\gtil^1_1 = \id_{SM}$ and
$\gtil^1_j = 0$ for $j\geq 2$.

\begin{proposition}\label{prop:end_is_tdga}
  The composition induces a pairing of twisted chain complexes
  \[ \enrhom_E(M,N) \tensor \enrhom_E(L,M) \to \enrhom_E(L,N) \] which is
  associative and unital. In particular, $\enrhom_E(M,M)$ is a tdga.
\end{proposition}
\begin{proof}
  Let $\mu$ be the composition. The only possibly non-trivial part is to
  verify the Leibniz-rules. If $i \geq 0$, $\ftil \in \enrhom_E(M,N)_{st}$
  and $\gtil \in \enrhom_E(L,M)_{pq}$, we have
\[ \begin{split}
&\mu (\eins \tensor d_i + d_i \tensor \eins) (\ftil \tensor \gtil) \\
=& (-1)^{\langle d_i,\ftil \rangle} \ftil \mtil^M_i \gtil 
- (-1)^{\langle d_i,\ftil \rangle + \langle \gtil, \mtil_i \rangle} \ftil
\gtil \mtil^L_i
+ \mtil^N_i \ftil \gtil 
- (-1)^{\langle \ftil, \mtil_i \rangle} \ftil \mtil^M_i \gtil \\
=&d_i \mu (\ftil \tensor \gtil).
\end{split}\]
\end{proof}

\begin{remark} \label{rem:enrichment_pitfall} The last proposition says
  that the category of modules over a \dainfalg is enriched in twisted
  chain complexes. The morphisms of the ordinary category underlying this
  $\tChxes_k$-category are not the \dainfalg module morphisms of Definition
  \ref{def:modulemaps}. Unlike the case of modules over
  $A_{\infty}$-algebras and their enrichment in chain complexes, not even
  all components $f_{ij}$ of an $E$-module map $f$ are $0$-cycles in the
  $\enrhom_E(M,N)$. The reason is that for $i \geq 1$, the $f_{ij}$ would
  have to lie in a negative degree with respect to the first grading. We
  resist from allowing a negative grading there, since this would for
  example cause problems when forming total complexes. Therefore, the
  notation $\enrhom$ might be slightly misleading, but we keep it in lack
  of a better alternative. Moreover, the induced maps on $\enrhom_E(M,N)$
  to be defined next are not a consequence of the enrichment.
\end{remark}

Given $E$-module maps $f \colon M \to M'$ and $h \colon N \to N'$, we
define
\[
f^*_i \colon \enrhom_E(M',N) \to \enrhom_E(M,N)[i,-i], \quad 
(f_i^* (\gtil^1))^1_v = 
\sum_{1 \leq j \leq v}(-1)^{\sclprod{\ftil_i}{\gtil}}
\gtil^1_j \ftil^j_{iv} 
\]
and
\[
(h_*)_i \colon \enrhom_E(M,N') \to \enrhom_E(M,N)[i,-i], \quad 
((h_*)_i(\gtil^1))^1_v = 
\sum_{1 \leq j \leq v}
\htil^1_{ij}\gtil^j_v. 
\]
For degree reasons, $f^*_i$ and $(h_*)_i$ are only defined on $\gtil$
in horizontal degree $\geq i$.

\begin{lemma} \label{lem:induced_maps}
Both $f^*$ and $h_*$ are maps of twisted chain complexes. If $M' = N$, then
$f^*$ is a left $\enrhom_E(N,N)$-module map. If $N'=M$, then $h_*$ is a
right $\enrhom_E(M,M)$-module map. 
\end{lemma} 
\begin{proof}
Treating $f^*$ first, we have to compare the image of $\gtil^1 \in
\enrhom_E(M',N)$ under $\sum (-1)^i f^*_i d_p$ and $\sum d_p
f^*_i$. Restricted to $SM \tensor (SE)^{\tensor w-1}$, the image of
$\gtil^1$ under both maps is
\[
\sum_{\substack{i+p=u;\, 1 \leq j \leq q \leq w}} \hspace{-1.5ex}
\left( (-1)^{\sclprod{\ftil_i}{\gtil}} 
\mtil^{N,1}_{pq} \gtil^q_j \ftil^{j}_{iw} +
(-1)^{\sclprod{\mtil_p}{\gtil}+\sclprod{\ftil_i}{\gtil}} 
\gtil^1_j \mtil^{M',j}_{pq}\ftil^q_{iw}\right).
\]
For $h_*$ we calculate
\[
\sum_{i+p=u} (d_p ((h_*)_i(\gtil^1)))^1_w 
= \sum_{i+p=u} (-1)^i( (h_*)_i(d_p(\gtil^1)))^1_w.
\] 
The statement about the tdga-module structures follows from the
associativity of the composition.
\end{proof}

For $F$-modules $M$ and $N$, a \dainfalg map $f \colon E \to F$ induces a
$k$-linear map $\res_f \colon \enrhom_F(M,N) \to \enrhom_E(M,N)$. On
$\gtil^1 \colon SM \tensor TSE[s,t] \to SN$, the restriction of
$(\res_f)_u(\gtil^1)$ to $SM \tensor (SE)^{\tensor v-1}$ is defined as
follows: for $v=1$ and $u=0$, it is $\gtil^1$, for $v=1$ and $u \neq 0$, it
is $0$, and for $v \geq 2$ it is
\[ (\res_f \gtil^1)^1_{uv} = \sum_{2\leq j \leq v}
(-1)^{\sclprod{\gtil}{\ftil_u}} 
 \gtil^1_j(\eins \tensor \ftil^{j-1}_{u,v-1}) 
 \colon SM \tensor (SE)^{\tensor v-1}[s,t] \to SN[u,-u].\] 
\begin{lemma}
  The $\res_f$ is a map of twisted chain complexes and respects the
  composition pairing. Particularly, it is a tdga map if $M=N$.
\end{lemma}
\begin{proof} 
Both
\[
\sum_{i+p=u} d_p (\res_f)_i = \sum_{i+p=u} (-1)^i (\res_f)_i d_p \; \text{ and
}\; \sum_{i+p=u} \mu ((\res_f)_i \tensor (\res_f)_p) = (\res_f)_u \mu 
\]
can be checked directly.
\end{proof}

Let $E$ be a \dainf{}-algebra. Viewing $E$ as an $E$-module,
$\enrhom_E(E,E)$ is an tdga by Proposition \ref{prop:end_is_tdga}. We want
to define a \dainfalg map $\psi^E \colon E \to \enrhom_E(E,E)$. By Lemma
\ref{lem:vari_map_rel}, it is enough to give maps $\psitil^{E,1}_{iv}
\colon (SE)^{\tensor v} \to S \enrhom_E(E,E)[i,-i]$. For $i \geq 0$ and
$v\geq 1$, let $\rhotil_{iv}^{E,1}$ be the map with source $SE^{\tensor v}$
being adjoint to the restriction of $\mtil^{E,1}_{i}$ to $SE^{\tensor v}
\tensor SE \tensor TSE$, and let $\psitil^{E,1}_{iv} = \sigma^{-1}_E
\rhotil^{E,1}_{ij}$.

The analogs in the world of $A_{\infty}$-algebras to the next three lemmas
are \cite[Lemme 4.1.1.6.b and Lemme 5.3.0.1]{Lefevere_sur}.

\begin{lemma}\label{lem:psi_map_mult}
  The $\psitil^{E,1}_{ij}$ constitute a \dainfalg map $\psitil^E \colon E
  \to \enrhom_E(E,E)$.
\end{lemma}
\begin{proof}
  It is by Lemma \ref{lem:vari_map_rel} sufficient to show
\[
\sum_{i+p=u} \dtil_p \psitil^{E,1}_{iv} +
\hspace{-1.5ex} \sum_{\substack{i_1+i_2=u;\, j_1+j_2=v}}\hspace{-1.5ex} 
\mutil (\psitil^{E,1}_{i_1j_1}
\tensor \psitil^{E,1}_{i_2 j_2}) 
= \hspace{-1.5ex} \sum_{\substack{1\leq j\leq v;\, i+p=u}}\hspace{-1.5ex}
(-1)^i \psitil_{ij}^{E,1} \mtil^{E,j}_{pv}
\]
for $v \geq 1$. This is equivalent to
\[\begin{split}
& \sum_{i+p=u} (-1)^{p+1} \sigma_E^{-1} d_p \rhotil^{E,1}_{iv} +
 \hspace{-1.5ex}\sum_{\substack{i_1+i_2=u;\, j_1+j_2=v}} \hspace{-1.5ex}
 (-1)^{i_1+1} \sigma_E^{-1} \mu (\rhotil^{E,1}_{i_1j_1}
\tensor \rhotil^{E,1}_{i_2 j_2})\\
=&  \hspace{-1.5ex}\sum_{\substack{1\leq j\leq v;\,i+p=u}}
(-1)^i \sigma_E^{-1} \rhotil_{ij}^{E,1} \mtil^{E,j}_{pv}.
\end{split}\] 
We cancel the $\sigma_E^{-1}$ out, restrict along the injection
$(SE)^{\tensor w} \to SE \tensor TSE$, and consider the adjoint maps
$(SE)^{\tensor v} \tensor (SE)^{\tensor w} \to SE[u,2-u]$ to get the
following equation for $v,w \geq 1$ which is equivalent to the last one.
\[ \begin{split}
&\hspace{-1.5ex} \sum_{\substack{1 \leq j \leq w; \, i+p=u}}\hspace{-1.5ex} 
(-1)^{p+1} \mtil^{E,1}_{p,j} ( \mtil^{E,1}_{i,v+w+1-j} \tensor
\eins^{\tensor j-1})\\
+&\hspace{-1.5ex}\sum_{\substack{1 \leq j \leq w; \, i+p=u}} \hspace{-1.5ex}
(-1)^{i+1} \mtil^{E,1}_{i,v+j} ( \eins^{\tensor v} \tensor
\mtil^{E,w}_{p,j}) \\
+& \sum_{\substack{i_1 + i_2 = u \\ j_1 + j_2 = v, \, j_s \geq 1 \\ 
w_1 + w_2 = w, \, w_s \geq 0}}
(-1)^{i_1+1} \mtil^{E,1}_{i_1, j_1 + 1 + w_1} (\eins^{\tensor j_1} \tensor
\mtil^{E,1}_{i_2, j_2 +1 + w_2} \tensor \eins^{\tensor w_1})\\
&= \hspace{-1.5ex}\sum_{\substack{1 \leq j \leq v; \, i+p = u}} \hspace{-1.5ex}
(-1)^i \mtil^{E,1}_{i,j+w}(\mtil^{E,j}_{p,v} \tensor \eins^{\tensor w})
\end{split} \] 
This holds since $E$ is an \dainf{}-algebra. 
\end{proof}

If $M$ is an $E$-module, there is a map $\psi^{M,E} \colon M \to
\enrhom_E(M,E)$ of twisted chain complexes: we define $\psi^{M,E}_{i}$ to
be the adjoint of the map
\[ (-1)^i \mtil^{M,1}_{i} ( \sigma_M^{-1} \tensor \eins \tensor \eins) \colon M \tensor SE
\tensor TSE \to SM[i,-i].\]
When forming the associated map $\psitil^{M,E}$, the $(-1)^i$ cancels, and
for $M = E$ we observe that $\psitil^{E,E}_i = \psitil^{E,1}_{i1}$. In
particular, the underlying map of twisted chain complexes of the
$\psi^E$ is $\psi^{E,E}$. 

The proof of the last lemma is easily adopted to provide\begin{lemma}\label{lem:psi_map_add}
  The $\psi^{M,E} \colon M \to \enrhom_E(E,M)$ is a map of twisted chain
  complexes.
\end{lemma}

\begin{lemma}\label{lem:map_E1equiv}
  The map $\psi^{M,E}$ of Lemma \ref{lem:psi_map_add} is an
  $E_1$-equivalence.
\end{lemma}
\begin{proof}
  We show that $\psi^{M,E}_0$ is a chain equivalence. For $\gtil^1 \in
  \enrhom_E(E,M)$,
  \[ \Hom_k(SE \tensor TSE, SM) \xrightarrow{\Psi_1(\eta)^*} \Hom_k(Sk,Sm)
  \xrightarrow{\Psi_1^{-1}} \Hom_k(k,M) \iso M\] can be used to define
  $\tau(\gtil^1) = (-1)^{\langle \gtil^1, \psi \rangle} \Psi^{-1}_1
  (\gtil^1 \Psi_1(\eta) )$.  The $\tau$ is a chain map with respect to
  $d_0$ and $\tau \psi_0^{M,E} = \id$.

  Next we define a chain homotopy between $\psi^{M,E}_0 \tau$ and $\id$ on
  $\enrhom_E(E,M)$. Let $h$ be the composite $SE \tensor TSE[0,1] \iso Sk
  \tensor SE \tensor TSE \to SE \tensor TSE$ given by $S \eta$ on $Sk$ and
  the inclusion $SE \tensor TSE \to TSE$.  We define $H \colon
  \enrhom_E(E,M) \to \enrhom_E(E,M)[0,-1]$ by $H(\gtil^1) = (-1)^{\langle
    \gtil^1 , h \rangle} \gtil^1 h$ and calculate
\[
(d_0 H + Hd_0)(\gtil^1) = \gtil^1 - \psi^{M,E} \tau
(\gtil^1).\]
Therefore, $\psi^{M,E}$ induces an isomorphism on homology with respect to
$d_0$.
\end{proof}

\begin{lemma} \label{lem:flowerstar_homotopy} The maps $(f_*)_0 \psi^E_0$
  and $\psi_0^{F,E} f_0$ are chain homotopic as $d_0$-chain maps from $E$
  to $\enrhom_E(E,F)$.
\end{lemma}
\begin{proof}
The chain homotopy is 
\[ H \colon E \to \enrhom_E(E,F), \quad (H(x))^1_j =
\ftil^1_{0,j+1}(\sigma_E^{-1}x \tensor \eins^{\tensor j}).\]
\end{proof}
\begin{corollary} \label{cor:flowerstar_equiv} Let $f \colon E \to F$ be a
  map of \dainf{}-algebras. The induced map $f_* \colon \enrhom_E(E,E) \to
  \enrhom_E(E,F)$ is an $E_2$-equivalence of twisted chain complexes if and
  only if $f$ is an $E_2$-equivalence.
\end{corollary}
\begin{proof}
  By Lemma \ref{lem:flowerstar_homotopy}, the maps induced by $f_* \psi^E$
  and $\psi^{E,F} f$ on $H_*^h H_*^v$ coincide. Both $\psi^E$ and
  $\psi^{E,F}$ are $E_1$-equivalences by Lemma \ref{lem:map_E1equiv}, and
  hence $E_2$-equivalences.
\end{proof}

\begin{lemma}\label{lem:psiFE_factorization}
  Let $f \colon E \to F$ be a \dainf{}-algebra map.  Then $\psi^{F,E}$ and
  $f^* \res_f \psi^{F,F}$ coincide as maps of twisted chain complexes from
  $F$ to $\enrhom_E(E,F)$.
\end{lemma}
\begin{proof}
One can check that
\[ ((f^* \res_f)_i (\gtil^1))^1_j = \sum_{1 \leq q \leq j}
(-1)^{\sclprod{\gtil}{\ftil_i}} \gtil^1_q \ftil^q_{ij}
\]
holds for $j \geq 1$. 
This easily implies the assertion.
\end{proof}
\begin{corollary} \label{cor:fupperstar_res_equiv} The map $f^* \res_f$ is
  an $E_2$-equivalence of twisted chain complexes.
\end{corollary}

\begin{proposition} \label{prop:qequiv} Let $f \colon E \to F$ be an
  $E_2$-equivalence of \dainfalgs with $E_2$-homology concentrated in
  horizontal degree $0$. Then the dgas $\Tot \enrhom_E(E,E)$ and $\Tot
  \enrhom_F(F,F)$ are quasi-isomorphic.
\end{proposition}
\begin{proof}
  By the Corollaries \ref{cor:flowerstar_equiv} and
  \ref{cor:fupperstar_res_equiv}, there is a diagram
  \[ \enrhom_E(E,E) \xrightarrow{f_*} \enrhom_E(E,F) \xleftarrow{f^*
    \res_f} \enrhom_F(F,F) \] with both maps $E_2$-equivalences. Setting
  \[R = \Tot \enrhom_E(E,E), \quad S = \Tot \enrhom_F(F,F), \; \text{ and }
  \quad X = \Tot \enrhom_E(E,F),\] we obtain dgas $R$ and $S$, an
  $S$-$R$-bimodule $X$, a quasi-isomorphism $\alpha \colon R \to X$ of
  right $R$-modules and an quasi-isomorphism $\beta \colon S \to X$ of left
  $S$-modules. This data is compatible with the unit maps in that the
  square obtained from composing the unit maps of $R$ and $S$ with $\alpha$
  and $\beta$ commutes.

  The described data is almost a `quasi-equivalence' of dgas as described
  in \cite[p.31]{Keller_deriving} and generalized in \cite[Definition
  A.2.1]{Schwede_S_modules}. The missing part is an element of $X$ such
  that left and right multiplication with the element equals $\alpha$ and
  $\beta$, respectively (compare Remark \ref{rem:enrichment_pitfall}).
  However, we can follow closely the strategy of the proof of \cite[Lemma
  A.2.3]{Schwede_S_modules}, exploiting the (projective) model structures
  on $\Chxes_k$ and $\Mod$-$R$.

  We factor $\alpha$ in the model category of right $R$-modules as a
  composition of an acyclic cofibration $\xymatrix@1{\alpha' \colon R \,
    \ar@{>->}[r]^{\sim} & Y}$ and an acyclic fibration
  $\xymatrix@1{\alpha'' \colon Y \ar@{->>}[r]^{\sim} & X}$. We consider the
  dga $T = \Hom_R(Y,Y)$, the $T$-$R$-bimodule $V = \Hom_R(R,Y)$, and the
  $S$-$T$-bimodule $W = \Hom_R(Y,X)$. The map $\alpha'$ induces an acyclic
  fibration $(\alpha')^* \colon T \to V$ and a quasi-isomorphism
  $(\alpha')_* \colon R \iso \Hom_R(R,R) \to V$. The multiplications on $R$
  and $T$ induce an associative and unital multiplication on the pullback
  $P(\alpha')$ of $(\alpha')_*$ and $(\alpha')^*$.  Hence the induced maps
  from $R \to P(\alpha') \ot T$ are maps of dgas. They are
  quasi-isomorphisms since $(\alpha')^*$ is an acyclic fibration.
 
  The $\alpha''$ induces an acyclic fibration $(\alpha'')_* \colon T \to W$
  since $Y$ is cofibrant as an $R$-module. Right multiplication with
  $\alpha''$ induces a map $S \to W$. It is a quasi-iso\-morphism since its
  composition with the quasi-isomorphism $W = \Hom_R(Y,X) \to \Hom_R(R,X)
  \iso X$ induced by $\alpha'$ equals $\beta$. The latter statement
  exploits the compatibility of $\alpha$ and $\beta$ with the units of $R$
  and $S$. As above, we form the pullback to get a chain $T \to P(\alpha'')
  \ot S$ of quasi-isomorphisms.
\end{proof}

\begin{proof}[Proof of Theorem \ref{thm:antimin_models}]
  Let $E$ be a \dainf{}-algebra. Lemma \ref{lem:Tot_mult} and Proposition
  \ref{prop:end_is_tdga} show that $\Tot \enrhom_E(E,E)$ is a dga.  If $A$
  is a dga and $E \to A$ an $E_2$-equivalence, Proposition
  \ref{prop:qequiv} implies that $\Tot \enrhom_A(A,A)$ is quasi-isomorphic
  to $\Tot \enrhom_E(E,E)$. Applying $\Tot$ to the $E_1$-equivalence $A \to
  \enrhom_A(A,A)$ and observing $A \iso \Tot A$ completes the proof.
\end{proof}
Over a field, every $A_{\infty}$-algebra is quasi-isomorphic to a dga.
However, not every \dainfalg $E$ is $E_2$-equivalent to a dga, since $E$
may have $E_2$-homology in horizontal degrees other than $0$. From this
perspective, the relation of dgas and \dainfalgs is more like viewing
ordinary modules as (positive) chain complexes concentrated in degree $0$.

\section{Four lemmas}
\begin{lemma} \label{lem:zln_cocycle} In the situation of the proof of
  Proposition \ref{prop:exis_str} we have $\mtil_{01}^{B,1} \ztil_{ln}=0$.
\end{lemma}

\begin{proof}
We start with calculating 
\begin{equation} \label{equ:zln_coc1}\begin{split} 
& \mtil^{B,1}_{01} \mtil^{B,1}_{02} \ftil_{ln}^2 
=- \mtil^{B,1}_{02}(\eins \tensor \mtil^{B,1}_{01} 
+ \mtil^{B,1}_{01} \tensor \eins) \ftil_{ln}^2 \\
=& \mtil^{B,1}_{11} \mtil^{B,1}_{02} \ftil^2_{l-1,n} 
- \hspace{-1.5ex} \sum_{\substack{i+p=l; \, 2 \leq j < n}} \hspace{-1.5ex} 
(-1)^i \mtil^{B,1}_{02} \ftil^2_{ij} \mtil^{E,j}_{pn}.
\end{split} \end{equation}
The second summand of $\ztil_{ln}$ contributes
\begin{equation} \label{equ:zln_coc2} \begin{split}
\mtil^B_{01}  \mtil^{B,1}_{11} \ftil^1_{l-1,n} =&
\mtil^{B,1}_{11} \mtil^{B,1}_{01} \ftil^1_{l-1,n} \\
=& -\mtil^{B,1}_{11} \mtil^{B,1}_{02}  \ftil^2_{l-1,n}
+\hspace{-1.5ex} \sum_{\substack{i+p=l-1; \,  2 \leq j < n}} \hspace{-1.5ex} 
(-1)^i \mtil^{B,1}_{11} \ftil^1_{ij} \mtil^{E,j}_{pn}.
\end{split} \end{equation}
Applying $\mtil_{01}^B$ to the last summand of $\ztil_{ln}$ gives
\begin{equation} \label{equ:zln_coc3} \begin{split}
&  - \hspace{-1.5ex}  \sum_{\substack{i+p=l; \, 2 \leq j < n}} \hspace{-1.5ex} 
(-1)^i \mtil^B_{01} \ftil^1_{ij}
\mtil^{E,j}_{pn} \\
=&  \hspace{-1.5ex}  \sum_{\substack{i+p=l; \, 2 \leq j < n}} \hspace{-1.5ex} 
(-1)^{i}  \mtil^{B,1}_{02}
\ftil^2_{ij} \mtil^{E,j}_{pn} 
- \hspace{-1.5ex}  \sum_{\substack{i+p=l-1; \, 2 \leq j < n}} \hspace{-1.5ex} 
(-1)^{i}  \mtil^{B,1}_{11}  \ftil^1_{ij} \mtil^{E,j}_{pn}.\\
\end{split} \end{equation}
Forming the sum of \eqref{equ:zln_coc1}, \eqref{equ:zln_coc2}, and
\eqref{equ:zln_coc3} shows the assertion. 
\end{proof}

\begin{lemma} \label{lem:zln_works} In the situation of the proof of
  Proposition \ref{prop:exis_str}, the $\mtil_{ln}^{E,1}$ together with the
  $\mtil^{E,1}_{ij}$ for $i+j < l+n$ satisfy \ref{equ:obj_rel} for $u+v
  \leq l+n+1$.
\end{lemma}
\begin{proof} 
Since $\mtil^{E, 1}_{01} =0$, we have to verify
\begin{equation}
\sum_{\substack{i+p = u \\i+j \geq 2; \,j<v}} (-1)^i \mtil^{E, 1}_{ij} \mtil^{E,j}_{pv} = 0
\end{equation}
for $u+v = l+n+1$. The term $\mtil^{E,1}_{ij} \mtil^{E,j}_{pv}$
is represented by $\ztil_{ij} \mtil^{E,j}_{pw}$ if $i+j > 2$, by
$\mtil^{B,1}_{02} (\ftil^1_{01} \tensor \ztil_{u,v-1})+\mtil^{B,1}_{02}
(\ztil_{u,v-1} \tensor \ftil^1_{01})$ if $i=0$ and $j=2$, and by
$\mtil^{B,1}_{11} \ztil_{u-1,v}$ if $i=1$ and $j=1$. We evaluate these
expressions using the already established relations of the
$\mtil^{E_j}_{pq}$ and the definition of $\ztil_{ij}$. First we
observe
\begin{equation} \label{equ:zln_right1} \begin{split}
- \mtil^{B,1}_{11} \ztil_{u-1,v} 
= - \mtil^{B,1}_{11} \mtil^{B,1}_{02} \ftil^2_{u-1,v} 
+ \hspace{-1.5ex}  \sum_{\substack{i+p=u-1; \,  i+j \geq 2}}   \hspace{-1.5ex} 
(-1)^i\mtil^{B,1}_{11} \ftil^1_{ij} \mtil^{E,j}_{pv}.
\end{split} \end{equation}
Next we deduce
\begin{equation} \label{equ:zln_right3} \begin{split}
&\mtil^{B,1}_{02} (\ftil^1_{01} \tensor  \ztil_{u,v-1}) 
+  \mtil^{B,1}_{02} ( \ztil_{u,v-1}  \tensor  \ftil^1_{01})  \\
&+ \sum_{\substack{i+p=u; \, j+q=v  \\ i+j \geq 2}} \left(
(-1)^i \mtil^{B,1}_{02} (\ftil^1_{ij} \tensor \ftil^1_{01} \mtil^{E,1}_{pq})
+ \mtil^{B,1}_{02} (\ftil^1_{01} \mtil^{E,1}_{pq} \tensor \ftil^1_{ij}) \right) \\
=& \sum_{\substack{i_1+i_2+p=u; \, j_1+q=v \\ i_2+j_2 \geq 2}} 
(-1)^{i_1+i_2+1}  \mtil^{B,1}_{02} (\ftil^1_{i_1 j_1} \tensor
\ftil^1_{i_2 j_2}) (\eins^{\tensor j_1} \tensor \mtil^{E,j_2}_{pq} ) \\
&+ \sum_{\substack{i_1+i_2+p=u; \, j_2+q=v \\ i_1 + j_1 \geq 2}}
(-1)^{i_1+i_2+1}  \mtil^{B,1}_{02} (\ftil^1_{i_1 j_1} \tensor
\ftil^1_{i_2 j_2}) (\mtil^{E,j_1}_{pq} \tensor \eins^{\tensor j_2}) \\
&+ \mtil^{B,1}_{11} \mtil^{B,1}_{02} \ftil^2_{u-1,v} + 
\sum_{\substack{i_1+i_2 =u;\, j_1 + j_2 = v \\ i_s +j_s \geq 2}}
\mtil^{B,1}_{01}  \mtil^{B,1}_{02} (\ftil^1_{i_1 j_1} \tensor
\ftil^1_{i_2 j_2}).
\end{split} \end{equation}

The last step is to calculate
\begin{equation} \label{equ:zln_right4} \begin{split}
&\hspace{-1.5ex} \sum_{\substack{i+p=u; \, i+j > 2}}  \hspace{-1.5ex}
(-1)^i \ztil_{ij} \mtil^{E,j}_{pv} \\
=& \sum_{\substack{i+p=u;\,j+q=v \\ i+j \geq 2}} \left(
(-1)^i \mtil^{B,1}_{02} (\ftil^1_{i j} \tensor \ftil^1_{01} \mtil^{E,1}_{pq} ) 
+  \mtil^{B,1}_{02} (\ftil^1_{01} \mtil^{E,1}_{pq} \tensor
\ftil^1_{ij}) \right)\\
&+ \sum_{\substack{i_1 +i_2 +p=u;\, j_1+q=v \\ i_2+j_2 \geq 2}} (-1)^{i_1 +i_2}  \mtil^{B,1}_{02} (\ftil^1_{i_1 j_1} \tensor \ftil^1_{i_2 j_2})
(\eins^{\tensor j_1} \tensor \mtil^{E,j_2}_{pq} ) \\
&+ \sum_{\substack{i_1 +i_2 +p=u;\, j_2+q=v \\ i_1+j_1 \geq 2}} (-1)^{i_1
  +i_2} \mtil^{B,1}_{02} (\ftil^1_{i_1 j_1} \tensor \ftil^1_{i_2 j_2})
 (\mtil^{E,j_1}_{pq} \tensor \eins^{\tensor j_2})\\
& - \hspace{-1.5ex} \sum_{\substack{i+p=u; \, i+j \geq 2}}\hspace{-1.5ex} 
(-1)^i \mtil^{B,1}_{11} \ftil^1_{ij} \mtil^{E,j}_{pv}. 
\end{split} \end{equation}
Adding up \eqref{equ:zln_right1},
\eqref{equ:zln_right3}, and \eqref{equ:zln_right4} shows that the sum the
representatives for the terms in the desired formula has values in
coboundaries.
\end{proof}

\begin{lemma} \label{lem:yln_cycle}
In the situation of the proof of Proposition \ref{prop:funct_str},
$m^{C,1}_{01} \ytil_{ln}=0$ holds. 
\end{lemma}

\begin{proof}
We apply $\mtil_{01}^{C,1}$ to the six terms of $\ytil_{ln}$. The first gives
\begin{equation} \label{equ:yln_cycleT1}
\begin{split}
& \mtil_{01}^{C,1}  \mtil_{11}^{C,1} \htil^1_{l-1,n} \\
=& - (-1)^l \mtil_{11}^{C,1} \gbetatil_{l-1,n}^{1} 
+ (-1)^l \mtil_{11}^{C,1} \alphaftil_{l-1,n}^{1}
+ (-1)^l \hspace{-1.5ex} \sum_{i+p=l-1} \hspace{-1.5ex} \mtil_{11}^{C,1} \htil_{ij}^{1}
\mtil^{E,j}_{pn} \\
&- \sum_{\substack{i+p=l-1\\j+q=n}} \mtil_{11}^{C,1} \mtil_{02}^{C,1} ( \gbetatil_{ij}^1
\tensor \htil^1_{pq})
- \sum_{\substack{i+p=l-1\\j+q=n}} (-1)^p \mtil_{11}^{C,1} \mtil_{02}^{C,1} ( \htil^1_{ij}
\tensor \alphaftil_{pq}^1)\end{split}
\end{equation}
The second gives
\begin{equation} \label{equ:yln_cycleT2}
\begin{split}
&(-1)^l \mtil_{01}^{C,1} \alphaftil_{ln}^1 \\
=& - (-1)^{l} \mtil_{11}^{C,1} \alphaftil_{l-1,n}^1 
 - (-1)^{l}  \mtil_{02}^{C,1} \alphaftil_{ln}^2 
+ \sum_{i+p=l} (-1)^p \alphaftil_{ij}^1 \mtil_{pn}^{E, j}
\end{split}
\end{equation}
The third gives
\begin{equation} \label{equ:yln_cycleT3}
\begin{split}
&(-1)^{l+1}  \sum_{\substack{i+p=l\\ i+j \geq 2}}
\mtil_{01}^{C,1} \gtil^{1}_{ij} \betatil^j_{pn} \\
=& (-1)^{l} \mtil_{11}^{C,1} \gbetatil_{l-1,n}^{1} 
+ (-1)^{l}  \mtil_{02}^{C,1} \gbetatil_{ln}^{2}
- \sum_{i+p=l} (-1)^p \gbetatil_{ij} \mtil^{E,j}_{pn}
\end{split}
\end{equation}
The fourth gives
\begin{equation} \label{equ:yln_cycleT4}
\begin{split}
&  (-1)^{l} \sum_{i+p=l} (-1)^i \mtil_{01}^{C,1} \htil^1_{ij} \mtil^{E,1}_{pn}  \\
=& \sum_{i+p=l} (-1)^p \gbetatil_{ij}^1 \mtil^{E,j}_{pn} 
- \sum_{i+p=l} (-1)^p \alphaftil_{ij}^1 \mtil_{pn}^{E, j} 
- (-1)^l \hspace{-1.5ex} \sum_{i+p=l-1} \hspace{-1.5ex} \mtil_{11}^{C,1} \htil_{ij}^{1}
\mtil^{E,j}_{pn} \\
&- (-1)^l \sum_{i_1+i_2 +p=l} \mtil_{02}^{C,1} (  \gbetatil_{i_1j_1}^1 \tensor 
\htil^1_{i_2 j_2}) \mtil_{pn}^{E, j_1+j_2} \\
&- (-1)^l \sum_{i_1+i_2 +p=l} (-1)^{i_2} \mtil_{02}^{C,1} ( \htil^1_{i_1 j_1}
\tensor \alphaftil_{i_2 j_2}^{1}) \mtil_{pn}^{E, j_1+j_2}\\
\end{split}
\end{equation}
The fifth gives
\begin{equation} \label{equ:yln_cycleT5}
\begin{split}
&\sum_{\substack{i+p=l\\j+q=n}} \mtil_{01}^{C,1} \mtil_{02}^{C,1} ( \gbetatil_{ij}^1
\tensor \htil^1_{pq})\\
=& \sum_{\substack{i+p=l-1\\j+q=n}} \mtil_{11}^{C,1} \mtil_{02}^{C,1} ( \gbetatil_{ij}^1
\tensor \htil^1_{pq}) 
+  (-1)^l \sum_{\substack{i+p=l\\j+q=n}} \mtil_{02}^{C,1} ( \gbetatil_{ij}^1
\tensor \alphaftil_{pq}^{1}) \\
&- (-1)^l \mtil_{02}^{C,1} \gbetatil_{ln}^{2}
+ (-1)^l \sum_{i_1+i_2 +p=l} \mtil_{02}^{C,1} (  \gbetatil_{i_1j_1}^1 \tensor 
\htil^1_{i_2 j_2}) \mtil_{pn}^{E, j_1+j_2}\\
&+  \sum_{\substack{i+p_1+p_2=l\\ j+q_1 +q_2=n}} (-1)^{p_2}
\mtil_{02}^{C,1} (\eins \tensor \mtil_{02}^{C,1}) (\gbetatil_{ij}^1 \tensor \htil^1_{p_1q_1}
\tensor \alphaftil_{p_2q_2}^{1})
\end{split}
\end{equation}
The sixth gives
\begin{equation} \label{equ:yln_cycleT6}
\begin{split}
&\sum_{\substack{i+p=l\\j+q=n}} (-1)^p \mtil_{01}^{C,1} \mtil_{02}^{C,1} ( \htil^1_{ij}
\tensor \alphaftil_{pq}^1)\\
=& \sum_{\substack{i+p=l-1\\j+q=n}} \hspace{-1ex}(-1)^p \mtil_{11}^{C,1} \mtil_{02}^{C,1} ( \htil^1_{ij}
\tensor \alphaftil_{pq}^{1})
- (-1)^l \sum_{\substack{i+p=l\\j+q=n}}\hspace{-1ex} \mtil_{02}^{C,1} ( \gbetatil_{ij}^1
\tensor \alphaftil_{pq}^{1})\\
&+ (-1)^l \mtil_{02}^{C,1} \alphaftil_{ln}^{2}
+ \hspace{-1.5ex} \sum_{i_1+i_2 +p=l} \hspace{-2.5ex} (-1)^{i_1 + p}\mtil_{02}^{C,1} ( \htil^1_{i_1 j_1}
\tensor \alphaftil_{i_2 j_2}^{1}) \mtil_{pn}^{E, j_1+j_2}\\
&+  \sum_{\substack{i_1+i_2+p=j\\ j_1+j_2+q=n}} (-1)^p
\mtil_{02}^{C,1} (\mtil_{02}^{C,1} \tensor \eins) (\gbetatil_{i_1 j_1}^1 \tensor \htil^1_{i_2 j_2}
\tensor \alphaftil_{pq}^{1})
\end{split}
\end{equation}
Indeed, the sum of \eqref{equ:yln_cycleT1}, \dots,  \eqref{equ:yln_cycleT6} is
zero. 
\end{proof}

\begin{lemma} \label{lem:yln_works} In the situation of the proof of
  Proposition \ref{prop:funct_str}, the $\betatil^1_{ln}$ represented by
  $y_{ln}$ satisfies \ref{equ:mor_rel} for $u+v=l+n+1$.
\end{lemma}

\begin{proof}
We have to show that $\beta$ is an \dainf{}-map, that is, 
\[ \sum_{i+p=u} (-1)^i \betatil^1_{ij} \mtil^{E,j }_{pv} = \sum_{i+p=u}
\mtil^{F, 1}_{ij} \betatil^j_{pv}. \]
We use the representing cocycles $(-1)^i\ytil_{ij}$ and $\ztil_{ij}$. 
A direct calculation shows
\[
\sum_{i+p=u} (-1)^i \ytil_{ij} \mtil^{E,j }_{pv} - \sum_{i+p=u}
\ztil_{ij} \betatil^j_{pv} 
= \mtil_{01}^{C,1} \gbetatil^1_{uv}- \sum_{i+p=u} (-1)^i \mtil^{C,1}_{01} \htil^{1}_{ij} \mtil^{E,j}_{pv}
\]
\end{proof}
\bibliographystyle{abbrv} 
\bibliography{dainf05}
\end{document}